\documentclass{article}
\usepackage[english]{babel}
\usepackage[utf8]{inputenc}
\usepackage{johd}
\usepackage{mathptmx}
\usepackage{amsmath}
\usepackage{mathtools}
\usepackage{verbatim}
\usepackage{appendix}
\usepackage{hyperref}

\numberwithin{equation}{section}

\usepackage{amsmath,amsthm}       
\usepackage{mathrsfs}      
\usepackage{amssymb}       
\usepackage{amsfonts}      

\usepackage{cancel}
\usepackage{indentfirst}

\usepackage{graphicx}
\usepackage{subfigure}
\usepackage{wrapfig}  
\usepackage{multicol}
\usepackage{tikz}     

\usepackage[framemethod=TikZ]{mdframed}
\usepackage{color}
\usepackage{authblk}
\usepackage{textcomp}

\newtheorem{thm}{Theorem}[section]
\newtheorem{lemma}[thm]{Lemma}
\newtheorem{definition}[thm]{Definition}

\newtheorem{prop}[thm]{Proposition}

\title{Propagation of chaos for mean-field reflected BSDEs with jumps}

\date{\today}
\author{Yiqing Lin}

\author{Kun Xu\footnote{Corresponding author. Email address: \url{1949101x_k@sjtu.edu.cn} (K. Xu).}}
\affil[1]{\small{School of Mathematical Sciences, Shanghai Jiao Tong University, 200240 Shanghai, China.}}

\begin{document}

\maketitle

\begin{abstract}
In this paper, we study a class of mean-field reflected backward stochastic differential equations (MF-RBSDEs) driven by a marked point process and also analyze MF-RBSDEs driven by a Poisson process. Based on a $g$-expectation representation lemma, we give the existence and uniqueness of the particle system of MF-RBSDEs driven by a marked point process under Lipschitz generator conditions and obtain a convergence result of this system. We also establish the well-posedness of the MF-RBSDEs driven by a Poisson process and the convergence rate of the corresponding particle system towards the solution to the MF-RBSDEs driven by a Poisson process under bounded terminal, bounded obstacle conditions.
\end{abstract}


\section{Introduction}
Pardoux and Peng \cite{pardoux1990adapted} first introduced backward stochastic differential equations (BSDEs) and obtained the existence and uniqueness for the case of Lipschitz continuous coefficients. From then on, many efforts have been made to relax the conditions on the generators for the well-posedness of adapted solutions. For instance, Kobylanski \cite{kobylanski2000backward} studied quadratic BSDEs for a bounded terminal value $\xi$ via an approximation procedure of the driver. Thereafter, the result was generated by Briand and Hu \cite{briand2006bsde, briand2008quadratic} for unbounded terminal value $\xi$ of some suitable exponential moments. In contrast, Tevzadze \cite{tevzadze2008solvability} proposed a fundamentally different approach by means of a fixed point argument. At the same time, BSDEs have received numerous developments in various fields of PDEs \cite{delbaen2015uniqueness}, mathematical
finance \cite{el1997backward}, stochastic optimal controls \cite{yong2012stochastic} and etc. 

Moreover, El Karoui et al. \cite{Karoui1997ReflectedSO} studied reflected backward stochastic differential equations (RBSDEs) related to an obstacle problem for PDEs. The solution $Y$ of a RBSDE is required to be above a given continuous process $L$, and the solution $(Y, Z, K)$ satisfies the so-called flat-off condition (or, Skorokhod condition):
$$
\int_0^T\left(Y_t-L_t\right) d K_t=0,
$$
where $K$ is an increasing process. El Karoui et al. \cite{Karoui1997ReflectedSO} proved the solvability of RBSDE with Lipschitz $f$ and square integrable terminal $\xi$.
After that, the result of \cite{Karoui1997ReflectedSO} is generated by  Kobylanski et al. \cite{kobylanski2002reflected} for quadratic RBSDEs with bounded terminal values and bounded obstacles. Lepeltier and Xu \cite{lepeltier2007reflected} constructed the existence of a solution with unbounded terminal values, but still with a bounded obstacle. Bayrakstar and Yao \cite{bayraktar2012quadratic} studied the well-posedness of quadratic RBSDE under unbounded terminal and unbounded obstacles.  

The generalizations of BSDEs from a Brownian framework to a setting with jumps have aroused a lot of attention. Li and Tang \cite{Tang_1994} and Barles, Buckdahn and Pardoux \cite{barles1997backward} obtained the well-posedness for Lipschitz BSDEs with jumps (BSDEJ). Since then, different kind of BSDEJs have been investigated by many researchers, see \cite{Becherer_2006,morlais2010new,antonelli2016solutions,cohen2015stochastic,kazi2015quadratic,Barrieu_2013,ngoupeyou2010optimisation,jeanblanc2012robust,karoui2016quadratic,kaakai2022utility}. 
The extension to the case of reflected BSDEs with jumps can be found in e.g. \cite{dumitrescu2016reflected,dumitrescu2015optimal,dumitrescu2016generalized,essaky2008reflected,hamadene2003reflected,hamadene2016reflected}.

In particular, a class of exponential growth BSDEs driven by a random measure associated with a marked point process as follows is investigated by many researchers. 
\begin{equation}
Y_t= \xi + \int_t^T f\left(t, Y_s, U_s\right)dA_s-\int_t^T \int_E U_s(e) q(d s,d e).
\end{equation}
Here $q$ is a compensated integer random measure corresponding to some marked point process $(T_n,\zeta_n)_{n\ge 0}$, and $A$ is  the dual predictable projection of the event counting process related to the marked point process. In this paper, $A$ is assumed to be a continuous and increasing process. The well-posedness of BSDEs driven by general marked point processes were investigated in \cite{Confortola2013,Becherer_2006,Confortola_2014,Confortola2016,Confortola_2018}. A more general BSDE with both Brownian motion diffusion term and a non-explosive marked point process with totally inaccessible jumps was studied in Foresta \cite{foresta2021optimal}. 

Motivated by the connection to control of McKean-Vlasov equation or mean-field games, see Carmona et al. \cite{carmona2013control} and Acciaio et al. \cite{acciaio2019extended}, mean-field BSDEs were introduced by Buckdahn, Djehiche, Li, Peng \cite{buckdahn2009mean} and Buckdahn, Li, Peng \cite{buckdahn2009mean1}. After that, reflected mean-field BSDEs has been considered by Li \cite{li2014reflected} and Djehiche et al. \cite{djehiche2019mean}. Recently, Djehiche, Dumitrescu, and Zeng \cite{Djehiche2021} studied mean-field reflected BSDEs with jumps and  RCLL  obstacle. In this paper, we consider the particle system of following type of mean field RBSDEs driven by a marked point process, which generalized the results of Brownian driven BSDEs studied in \cite{bayraktar2012quadratic} and Poisson driven BSDEJs attributed to e.g.  Hamad\'ene and Ouknine \cite{hamadene2003reflected,  hamadene2016reflected}, Djehiche,  Dumitrescu and Zeng \cite{Djehiche2021}. 
\begin{equation}
\label{reflected BSDE in intro}
\left\{\begin{array}{l}
Y_t=\xi+\int_t^T f(s, Y_s, U_s, \mathbb{P}_{Y_s}) d A_s +\int_t^T d K_s-\int_t^T \int_E U_s(e) q(d s, d e), \quad \forall t \in[0, T], \quad \mathbb{P} \text {-a.s. } \\
Y_t \geq h(t, Y_t, \mathbb{P}_{Y_t}), \quad \forall t \in[0, T], \quad \mathbb{P} \text {-a.s. } \\
\int_0^T\left(Y_{t }-h\left(t , Y_{t }, \mathbb{P}_{Y_{t }}\right)\right) d K_t=0, \quad \mathbb{P} \text {-a.s. }
\end{array}\right.
\end{equation}
Some related studies on doubly RBSDEJs can be found in Cr\'epey and Matoussi \cite{crepey2008reflected}. Moreover, RBSDEJs driven by Lévy process considered in for instance,  Ren and El Otmani \cite{ren2010generalized}, Ren and Hu \cite{ren2007reflected} and El Otmani \cite{el2009reflected} are also enlightening. 
Compared with the jump setting in e.g. Matoussi and Salhi \cite{matoussi2020generalized}, the  process $A$ is not necessarily absolutely continuous with respect to the Lebesgue measure. This type of RBSDEs has been investigated in Foresta \cite{foresta2021optimal}. The author established the well-posedness with Lipschitz drivers with the help of a fixed point argument. The extension to the case of reflected BSDEs with jumps can also be found in e.g. \cite{dumitrescu2016reflected,dumitrescu2015optimal,dumitrescu2016generalized,essaky2008reflected}.

The theory of propagation of chaos can be traced back to the work by Kac \cite{kac1956foundations} whose initial
aim was to investigate the particle system approximation of some nonlocal partial differential equations (PDEs) arising in thermodynamics. Kac’s intuition was put into firm mathematical ground notably by Henry P McKean \cite{mckean1967propagation}, Alain-Sol Sznitman \cite{sznitman1991topics} and J{\"u}rgen G{\"a}rtner \cite{gartner1988mckean}. Further development and applications of propagation of chaos theory can be found in \cite{jabin2016mean,lacker2018strong,shkolnikov2012large}. Besides, Buckdahn et al. \cite{buckdahn2009mean}, Hu, Ren and Yang 
\cite{hu2023principal}, Laurière and Tangpi \cite{lauriere2022backward} and Briand et al. \cite{briand2020forward} studied the limit theorems for weakly interacting particles whose dynamics is given by a system of BSDEs in the case of non-reflected BSDE driven by Brownian motion. Li \cite{li2014reflected} extends the results of \cite{buckdahn2009mean} to reflected BSDEs where the weak interaction enters only the driver, while the work by Briand and Hibon \cite{briand2021particles} considers a particular class of mean reflected BSDEs. Djehiche, Dumitrescu and Zeng \cite{Djehiche2021} establish a propagation of chaos result for weakly interacting nonlinear Snell envelopes which converge to a class of mean-field RBSDEs with jumps and right-continuous and left-limited obstacle, where the mean-field interaction in terms of the distribution of the $Y$-component of the solution enters both the driver and the lower obstacle.

Similar with \cite{Djehiche2021}, in this paper we first consider the following particle system:
\begin{equation}
\label{chaos RBSDE in intro}
\left\{\begin{aligned}
& Y_t^{i}= \xi^{i}+\int_t^T f\left(s, Y_s^{i}, U_s^{i, i}, L_n\left[\mathbf{Y}_s\right]\right) d A_s+K_T^{i}-K_t^{i} -\int_t^T \int_{E} \sum_{j=1}^n U_s^{i, j}(e) q^j(d s, d e), \quad \forall t \in[0, T], \quad \mathbb{P} \text {-a.s., } \\
& Y_t^{i} \geq h\left(t, Y_t^{i}, L_n\left[\mathbf{Y}_t\right]\right), \quad \forall t \in[0, T], \\
& \int_0^T\left(Y_{t }^{i}-h\left(t , Y_{t }^{i}, L_n\left[\mathbf{Y}_{t }\right]\right)\right) d K_t^{i}=0.
\end{aligned}\right.   
\end{equation}
and prove that the mean-field limit of the $N$-particle
system (\ref{chaos RBSDE in intro}) converges to the mean-field RBSDE (\ref{reflected BSDE in intro}) under the framework of marked point process which means the BSDE is driven by a random measure associated with a marked point process. Then we study a family of weakly interacting process 
\begin{equation}
\label{chaos RBSDE Poisson in intro}
\left\{\begin{aligned}
& Y_t^{i}= \xi^{i}+\int_t^T f\left(s, Y_s^{i}, U_s^{i, i}, L_n\left[\mathbf{Y}_s\right]\right) d s+K_T^{i}-K_t^{i} -\int_t^T \int_{E} \sum_{j=1}^n U_s^{i, j}(e) \tilde{\mu}^j(d s, d e), \quad \forall t \in[0, T], \quad \mathbb{P} \text {-a.s., }\\
& Y_t^{i} \geq h\left(t, Y_t^{i, n}, L_n\left[\mathbf{Y}_t\right]\right), \quad \forall t \in[0, T],\\
& \int_0^T\left(Y_{t }^{i}-h\left(t , Y_{t }^{i}, L_n\left[\mathbf{Y}_{t }\right]\right)\right) d K_t^{i}=0.
\end{aligned}\right.   
\end{equation}
and show the convergence rate of the solution of this system (\ref{chaos RBSDE Poisson in intro}) to the solution of the corresponding mean-field RBSDE under bounded terminal and obstacle condition.   

The rest of this paper is organized as follows. In Section \ref{sec P}, we present some preliminary notations. In Section \ref{sec chaos A}, we prove the existence and uniqueness of system (\ref{chaos RBSDE in intro}) and give the convergence result of system (\ref{chaos RBSDE in intro}). In Section \ref{sec cvg rate}, we show the convergence rate of the solution of this system (\ref{chaos RBSDE Poisson in intro}) to the solution of the corresponding mean-field RBSDE.

\section{Preliminaries}
\label{sec P}
First of all, we need to explain the settings and notations in this paper. We call the BSDE driven by a marked point process as the marked point process framework, and the BSDE driven by Poisson process as the Poisson process framework. Because the above two frameworks are both involved in this paper, we use the same notations to represent some definitions under the two frameworks to simplify the symbols without causing ambiguity.

\subsection{General setting}
Under the MPP framework, we introduce some notions about marked point processes and some basic assumptions. More details about marked point processes can be found in \cite{foresta2021optimal, Bremaud1981, last1995marked, cohen2012existence}. We assume that $(\Omega, \mathscr{F}, \mathbb{P})$ is a complete probability space and $E$ is a Borel space. We call $E$ the mark space and  $\mathscr{E}$ is its Borel $\sigma$-algebra. Given a sequence of random variables $(T_n,\zeta_n)$ taking values in $[0,\infty]\times E$, set $T_0=0$ and $\mathbb P-a.s.$ 
\begin{itemize}
\item $T_n\le T_{n+1},\ \forall n\ge 0;$
\item $T_n<\infty$ implies $T_n<T_{n+1} \ \forall n\ge 0.$
\end{itemize}
The sequence $(T_n,\zeta_n)_{n\ge 0}$ is called a marked point process (MPP). Moreover, we assume the marked point process is non-explosive, i.e., $T_n\to\infty,\ \mathbb P-a.s.$

Define a random discrete measure $p$ on $((0,+\infty) \times E, \mathscr{B}((0,+\infty) \otimes \mathscr{E})$ associated with each MPP:
\begin{equation}
\label{eq p}
    p(\omega, D)=\sum_{n \geq 1} \mathbf{1}_{\left(T_n(\omega), \zeta_n(\omega)\right) \in D} .
\end{equation} 
For each $\tilde C \in \mathscr{E}$, define the counting process $N_t(\tilde C)=p((0, t] \times \tilde C)$ and denote $N_t=N_t(E)$. Obviously, both are right continuous increasing process starting from zero.  Define for $t \geq 0$
$$
\mathscr{G}_t^0=\sigma\left(N_s(\tilde C): s \in[0, t], \tilde C \in \mathscr{E}\right)
$$
and $\mathscr{G}_t=\sigma\left(\mathscr{G}_t^0, \mathscr{N}\right)$, where $\mathscr{N}$ is the family of $\mathbb{P}$-null sets of $\mathscr{F}$. 
Note by $\mathbb{G}=\left(\mathscr{G}_t\right)_{t \geq 0}$ the  completed filtration generated by the MPP, which is right continuous and satisfies the usual hypotheses.  Given a standard Brownian motion $W\in \mathbb R^d$, independent with the MPP, let $\mathbb{F}:=\left\{\mathcal{F}_t\right\}_{t \in[0, T]}$ be the completed filtration generated by the MPP and $W$, which satisfies the usual conditions as well.

Each marked point process has a unique compensator $\lambda$, a predictable random measure such that
$$
\mathbb{E}\left[\int_0^{+\infty} \int_E C_t(e) p(d t, d e)\right]=\mathbb{E}\left[\int_0^{+\infty} \int_E C_t(e) \lambda(d t, d e)\right]
$$
for all $C$ which is non-negative and $\mathscr{P}^{\mathscr{G}} \otimes \mathscr{E}$-measurable, where $\mathscr{P}^{\mathscr{G}}$ is the $\sigma$-algebra generated by $\mathscr{G}$-predictable processes. Moreover, in this paper we always assume that there exists a function $\phi$ on $\Omega \times[0,+\infty) \times \mathscr{E}$ such that $\lambda(\omega, d t d e)=\phi_t(\omega, d e) d A_t(\omega)$, where $A$ is the dual predictable projection of $N$. In other words, $A$ is the unique right continuous increasing process with $A_0=0$ such that, for any non-negative predictable process $D$, it holds that,
$$
\mathbb E\left[\int_{0}^\infty D_t dN_t\right]=E\left[\int_{0}^\infty D_t dA_t\right].
$$

Fix a terminal time $T>0$, we can define the integral
$$
\int_0^T \int_E C_t(e) q(d t d e)=\int_0^T \int_E C_t(e) p(d t d e)-\int_0^T \int_E C_t(e) \phi_t(d e) d A_t,
$$
under the condition
$$
\mathbb{E}\left[\int_0^T \int_E\left|C_t(e)\right| \phi_t(d e) d A_t\right]<\infty .
$$
Indeed, the process $\int_0^{\cdot} \int_E C_t(e) q(d t, d e)$ is a martingale. Note that  $\int_a^b$ denotes an integral on $(a, b]$ if $b<\infty$, or on $(a, b)$ if $b=\infty$.

Under the Poisson process framework, let $\left(\Omega, \mathcal{F},\left\{\mathcal{F}_t\right\}_{0 \leq t \leq T}, \mathbb{P}\right)$ be a filtered probability space, whose filtration satisfies the usual hypotheses of completeness and right-continuity. We suppose that this filtration is generated by by the following two mutually independent processes:
\begin{itemize}
    \item a $d$-dimensional standard Brownian motion $\left\{B_t\right\}_{t \geq 0}$, and
    \item a Poisson random measure $\mu$ on $\mathbb{R}_{+} \times E$, where $E \triangleq \mathbb{R}^{\ell} \backslash\{0\}$ is equipped with its Borel field $\mathcal{E}$, with compensator $\lambda(\omega, d t, d e)$. We assume in all the paper that $\lambda$ is absolutely continuous with respect to the Lebesgue measure $d t$, i.e. $\lambda(\omega, d t, d e)=\nu_t(\omega, d e) d t$. Finally, we denote $\tilde{\mu}$ the compensated jump measure
    $$
    \widetilde{\mu}(\omega, d e, d t)=\mu(\omega, d e, d t)-\nu_t(\omega, d e) d t .
    $$
\end{itemize}

Denote by $\mathbb{F}:=\left\{\mathcal{F}_t\right\}_{t \in[0, T]}$ the completion of the filtration generated by $\tilde{\mu}$. 
Following Li and Tang \cite{Tang_1994} and Barles, Buckdahn and Pardoux \cite{barles1997backward}, the definition of a BSDE with jumps is then
\begin{definition}
Let $\xi$ be a $\mathcal{F}_T$-measurable random variable. A solution to the BSDEJ with terminal condition $\xi$ and generator $f$ is a triple $(Y, Z, U)$ of progressively measurable processes such that
\begin{equation}
\label{eq defn}
Y_t=\xi+\int_t^T f_s\left(Y_s, Z_s, U_s\right) d s-\int_t^T Z_s d B_s-\int_t^T \int_E U_s(x) \widetilde{\mu}(d s, d e), t \in[0, T], \mathbb{P}-a . s .,
\end{equation}
where $f: \Omega \times[0, T] \times \mathbb{R} \times \mathbb{R}^d \times \mathcal{A}(E) \rightarrow \mathbb{R}$ is a given application and
$$
\mathcal{A}(E):=\{u: E \rightarrow \mathbb{R}, \mathcal{B}(E)-\text { measurable }\} .
$$

Then, the processes $Z$ and $U$ are supposed to satisfy the minimal assumptions so that the quantities in (\ref{eq defn}) are well defined:
$$
\int_0^T\left|Z_t\right|^2 d t<+\infty,\left(\operatorname{resp} . \int_0^T \int_E\left|U_t(x)\right|^2 \nu_t(d x) d t<+\infty\right), \mathbb{P}-\text { a.s. }
$$
\end{definition}

As the Brownian motion $B$ and the integer valued random measure $\mu$ is independent, proof related to the two does not interfere with each other. In order to highlight the key points of this article, we only consider equation without the Brownian motion term.

\subsection{Notation}

We denote a generic constant by $C$, which may change line by line, is sometimes associated with several subscripts (such as $C_{K,T}$ ) showing its dependence when necessary. Let us introduce the following spaces for stochastic processes:
\begin{itemize}
    \item For any real $p \geq 1, S^p$ denotes the set of real-valued, adapted and c\`adl\`ag processes $\left\{Y_t\right\}_{t \in[0, T]}$ such that
    $$
    \|Y\|_{S^p}:=\mathbb{E}\left[\sup _{0 \leq t \leq T}\left|Y_t\right|^p\right]^{1 / p}<+\infty.
    $$
    Then $\left(S^p,\|\cdot\|_{ S^p}\right)$ is a Banach space.
    \item For any real $p \geq 1,\ {L}^p$ denotes the set of real-valued, $\mathcal{F}_T$-measurable random variables $\xi$ such that,
    $$
    \|\xi\|_{{L}^p}:=\mathbb{E}\left[\left|\xi\right|^p\right]^{1 / p}<+\infty.
    $$
    \item For any $p \geq 1$, we denote by $\mathcal{E}^p$ the collection of all stochastic processes $Y$ such that $e^{|Y|} \in$ $S^p$.   We write $Y \in \mathcal{E}$ if $Y \in \mathcal{E}^p$ for any $p \geq 1$, $\mathcal{L}$ is defined in a similar way.
    \item $\mathscr{M}^{2,p}$ the set of predictable processes $U$ such that $$\|U\|_{\mathscr{M}^{2,p}}:=\left(\mathbb{E}\left[\int_{[0, T]} \int_E\left|U_s(e)\right|^2 \nu_s(de) d s\right ]^{\frac{p}{2}}\right)^{\frac{1}{p}}<\infty.$$
    \item $H^{2,p}_{\nu}$ is the space of predictable processes $U$ such that 
    $$
    \|U\|_{H_\nu^{2,p}}:=\left(\mathbb{E}\left[\int_{[0, T]} \int_E\left|U_s(e)\right|^2 \phi_s(de) d A_s\right ]^{\frac{p}{2}}\right)^{\frac{1}{p}}<\infty.
    $$
    \item $L^0\left(\mathscr{B}(E)\right)$ denotes the space of $\mathcal B(E)$-measurable functions.  For $u \in L^0\left(\mathscr{B}(E)\right)$, define
    $$
    L^2(E,\mathcal{B}(E),\phi_t(\omega,dy)):=\left\{u \in L^0\left(\mathscr{B}(E)\right):\left\|u\right\|_t:=\left(\int_E\left|u(e)\right|^2  \phi_t(d e)\right)^{1 / 2}<\infty\right\},
    $$
    $$
    L^2(E,\mathcal{B}(E),\nu_t(\omega,dy)):=\left\{u \in L^0\left(\mathscr{B}(E)\right):\left|u\right|_\nu:=\left(\int_E\left|u(e)\right|^2  \nu_t(d e)\right)^{1 / 2}<\infty\right\}.
    $$
    \item $\mathcal{S}^\infty$ is the space of $\mathbb{R}$-valued c\`adl\`ag and $\mathbb F$-progressively measurable processes $Y$ such that
    $$
    \|Y\|_{\mathcal{S}^\infty}:=\left\|\sup _{0 \leq t \leq T}|Y_t|\right\|_{\infty}<+\infty.
    $$
    \item $\mathcal{J}^{\infty}$ is the space of functions which are $d \mathbb{P} \otimes v(d z)$ essentially bounded i.e.,
    $$
    \|\psi\|_{\mathcal{J}^{\infty}}:=\left\|\sup _{t \in[0, T]}\| \psi_t\|_{\mathcal{L}^{\infty}(v)}\right\|_{\infty}<\infty,
    $$
    where $\mathcal{L}^{\infty}(v)$ is the space of $\mathbb{R}^k$-valued measurable functions $v(d z)$-a.e. bounded endowed with the usual essential sup-norm.
    \item $\mathcal{P}_p(\mathbb{R})$ is the collection of all probability measures over $(\mathbb{R}, \mathcal{B}(\mathbb{R}))$ with finite $p^{\text {th }}$ moment, endowed with the $p$-Wasserstein distance $W_p$;
    \item $\mathcal{A}^D$ is the set of non-decreasing processes $K=\left(K_t\right)_{0 \leq t \leq T}$ starting from the origin, i.e. $K_0 = 0$;
\end{itemize}

For $\beta>0$, we introduce the following spaces.
\begin{itemize}
    \item $\mathcal{S}_\beta^p$ is the set of real-valued càdlàg adapted processes $y$ such that $\|y\|_{\mathcal{S}_\beta^p}^p:=\mathbb{E}\left[\sup _{0 \leq u \leq T} e^{\beta p s}\left|y_u\right|^p\right]<\infty$. We set $\mathcal{S}^p=\mathcal{S}_0^p$.
    \item $\mathbb{L}_\beta^p$ is the set of real-valued càdlàg adapted processes $y$ such that $\|y\|_{\mathbb{L}_\beta^p}^p:=\sup _{\tau \in \mathcal{T}_0} \mathbb{E}\left[e^{\beta p \tau}\left|y_\tau\right|^p\right]<\infty$, where $\mathcal{T}_t$ is the set of $\mathbb{F}$-stopping times $\tau$ such that $\tau \in[t, T]$ a.s. $\mathbb{L}_\beta^p$ is a Banach space. We set $\mathbb{L}^p=\mathbb{L}_0^p$.
\end{itemize}

\section{Well-posedness and convergence result of MF-RBSDEs driven by a MPP}
\label{sec chaos A}
In this section, we study the following MF-RBSDE driven by a MPP.
\begin{equation}
\label{reflected BSDE}
\left\{\begin{array}{l}
Y_t=\xi+\int_t^T f(s, Y_s, U_s, \mathbb{P}_{Y_s}) d A_s +\int_t^T d K_s-\int_t^T \int_E U_s(e) q(d s, d e), \quad \forall t \in[0, T], \quad \mathbb{P} \text {-a.s., } \\
Y_t \geq h(t, Y_t, \mathbb{P}_{Y_t}), \quad \forall t \in[0, T],\\
\int_0^T\left(Y_{t}-h\left(t, Y_{t}, \mathbb{P}_{Y_{t}}\right)\right) d K_t=0, \quad \mathbb{P} \text {-a.s. }
\end{array}\right.
\end{equation}
Next, we will now discuss the interpretation of (\ref{reflected BSDE}) at the particle level and study the well-posedness of the associated particle system.

Consider a family of weakly interacting processes $\mathbf{Y}:=\left(Y^{1}, \ldots, Y^{n}\right)$ evolving backward in time as follows: for $i=1, \ldots, n$,
\begin{equation}
\label{chaos RBSDE}
\left\{\begin{aligned}
& Y_t^{i}= \xi^{i}+\int_t^T f\left(s, Y_s^{i}, U_s^{i, i}, L_n\left[\mathbf{Y}_s\right]\right) d A_s+K_T^{i}-K_t^{i} -\int_t^T \int_{E} \sum_{j=1}^n U_s^{i, j}(e) q^j(d s, d e), \quad \forall t \in[0, T], \quad \mathbb{P} \text {-a.s., } \\
& Y_t^{i} \geq h\left(t, Y_t^{i}, L_n\left[\mathbf{Y}_t\right]\right), \quad \forall t \in[0, T], \\
& \int_0^T\left(Y_{t}^{i}-h\left(t, Y_{t}^{i}, L_n\left[\mathbf{Y}_{t}\right]\right)\right) d K_t^{i}=0,
\end{aligned}\right.   
\end{equation}
where the empirical measure associated to $\mathbf{Y}$ is denoted by
$$
L_n[\mathbf{Y}]:=\frac{1}{n} \sum_{k=1}^n \delta_{Y^k}
$$
and ${\xi^i}_{1 \leq i \leq N}$, ${f^i}_{1 \leq i \leq N}$ and ${q^i}_{1 \leq i \leq N}$ are independent copied of $\xi, f$ and $q$. Denote by $\mathbb{F}^n:=\left\{\mathcal{F}_t^n\right\}_{t \in[0, T]}$ the completion of the filtration generated by $\left\{q^i\right\}_{1 \leq i \leq n}$. Let $\mathcal{T}_t^n$ be the set of $\mathbb{F}^n$ stopping times with values in $[t, T]$.

\subsection{Well-posedness result}
We first introduce some basic assumptions that run through this section.

\hspace*{\fill}\\
\hspace*{\fill}\\
\noindent(\textbf{H1}) The process $A$ is continuous with $\|A_T\|_\infty<\infty$.
\hspace*{\fill}\\
\hspace*{\fill}\\

The first assumption is on the dual predictable projection $A$ of the counting process $N$ relative to $p$. We would like to emphasize that for $A_t$, we do not require absolute continuity with respect to the Lebesgue measure.

\hspace*{\fill}\\
\hspace*{\fill}\\
\hspace*{\fill}\\
\noindent(\textbf{H2})
For every $\omega \in \Omega,\ t \in[0, T],\ r \in \mathbb{R}$, $\mu \in \mathcal{P}_2(\mathbb{R})$ the mapping
$
f(\omega, t, r, \cdot, \mu):L^2(E,\mathcal{B}(E),\phi_t(\omega,dy)) \rightarrow \mathbb{R}
$
satisfies:
for every $U \in {H_\nu^{2,2}}$,
$$
(\omega, t, r, \mu) \mapsto f\left(\omega, t, r, U_t(\omega, \cdot), \mu \right)
$$
is Prog $\otimes \mathscr{B}(\mathbb{R})$-measurable.

\hspace*{\fill}\\
\hspace*{\fill}\\
\noindent(\textbf{H3})

\textbf{(a) (Continuity condition)} 
For every $\omega \in \Omega, t \in[0, T], y \in \mathbb{R}$, $u\in L^2(E,\mathcal{B}(E),\phi_t(\omega,dy))$, $\mu \in \mathcal{P}_2(\mathbb{R})$,  $(y, u, \mu) \longrightarrow f(t, y, u, \mu)$ is continuous. 
\hspace*{\fill}\\
\hspace*{\fill}\\

\textbf{(b) (Lipschitz condition)}  
There exists $C_f\geq 0$, such that for every $\omega \in \Omega,\ t \in[0, T],\ y_1, y_2 \in \mathbb{R}$,\ $u\in L^2(E,\mathcal{B}(E),\phi_t(\omega,dy))$, $\mu_1, \mu_2 \in \mathcal{P}_2(\mathbb{R})$, we have
$$
\begin{aligned}
& \left|f(\omega, t, y_1, u_1, \mu_1)-f\left(\omega, t, y_2, u_2, \mu_2\right)\right| \leq C_f\left(\left|y_1-y_2\right|+ \left\|u_1-u_2 \right\|_t+\mathcal{W}_2\left(\mu_1, \mu_2\right)\right).
\end{aligned}
$$
\hspace*{\fill}\\

\textbf{(c) (Growth condition)}
For all $t \in[0, T], \ (y,u) \in \mathbb{R} \times  L^2(E,\mathcal{B}(E),\phi_t(\omega,dy)),\mu \in \mathcal{P}_2(\mathbb{R})$: $\mathbb{P}$-a.s., there exists $\lambda>0$ such that,
$$
\underline{q}(t, y, u)=-\frac{1}{\lambda} j_{\lambda}(t,- u)-\alpha_t-\beta\left(|y|+ \mathcal{W}_2\left(\mu, \delta_0\right)\right) \leq f(t, y, u, \mu) \leq \frac{1}{\lambda} j_{\lambda}(t, u)+\alpha_t+\beta\left(|y|+\mathcal{W}_2\left(\mu, \delta_0\right)\right)=\bar{q}(t, y, u),
$$
where $\{\alpha_t\}_{0 \leq t \leq T}$ is  a progressively measurable nonnegative stochastic process.
\hspace*{\fill}\\
\hspace*{\fill}\\

\textbf{(d) (Integrability condition)}
We assume necessarily,
\begin{equation}
\label{Integrability condition}
\mathbb{E}\left[\int_0^T\alpha_s^2dA_s\right]<\infty.
\end{equation}
\hspace*{\fill}\\

\textbf{(e) (Convexity/Concavity condition)}  
For each $(t, y) \in[0, T] \times \mathbb{R}$, $u\in L^2(E,\mathcal{B}(E),\phi_t(\omega,dy)), \ \mu \in \mathcal{P}_2(\mathbb{R}), \ u \rightarrow f(t, y, u, \mu)$ is convex or concave.

\hspace*{\fill}\\
\hspace*{\fill}\\
\noindent(\textbf{H4})
$h$ is a mapping from $[0, T] \times \Omega \times \mathbb{R} \times \mathcal{P}_2(\mathbb{R})$ into $\mathbb{R}$ such that
\hspace*{\fill}\\
\hspace*{\fill}\\

\textbf{(a) (Continuity condition)} 
For all $(y, \mu) \in \mathbb{R} \times \mathcal{P}_2(\mathbb{R}),\ h(\cdot, y, \mu)$ is a continuous process.
\hspace*{\fill}\\
\hspace*{\fill}\\

\textbf{(b) (Lipschitz condition)}
$h$ is Lipschitz w.r.t. $(y, \mu)$ uniformly in $(t, \omega)$, i.e. there exists two positive constants $\gamma_1$ and $\gamma_2$ such that $\mathbb{P}$-a.s. for all $t \in[0, T]$,
$$
\left|h\left(t, y_1, \mu_1\right)-h\left(t, y_2, \mu_2\right)\right| \leq \gamma_1\left|y_1-y_2\right|+\gamma_2 \mathcal{W}_2\left(\mu_1, \mu_2\right)
$$
for any $y_1, y_2 \in \mathbb{R}$ and $\mu_1, \mu_2 \in \mathcal{P}_2(\mathbb{R})$.
\hspace*{\fill}\\
\hspace*{\fill}\\

\textbf{(c)}
For any $t \in [0,T]$, $h(t,0,\delta_0) \in S^2$.
\hspace*{\fill}\\
\hspace*{\fill}\\

\textbf{(d)}
The final condition $\xi^i\in L^2$ is $\mathscr{F}_T^n$-measurable, $i=1, \ldots, n$, and satisfies $\xi^i \geq h\left(T, \xi^{i}, L_n\left[\xi^n\right]\right)$ a.s. $i=1, \ldots, n$.
\hspace*{\fill}\\
\hspace*{\fill}\\

\noindent\textbf{(H5) (Uniform linear bound condition )
There exists a positive constant $C_0$ such that for each $t \in[0, T]$, $u\in L^2(E,\mathcal{B}(E),\phi_t(\omega,dy))$, if 
 $f$ is convex (resp. concave) in $u$, then $f(t,0,u,\mu)-f(t,0,0,\mu)\ge -C_0\|u\|_t$ (resp. $f(t,0,u,\mu)-f(t,0,0,\mu)\le C_0\|u\|_t$ ).
}

\hspace*{\fill}\\

Note that we have the inequality
\begin{equation}
\label{eq wp}
\mathcal{W}_p^p\left(L_n[\mathbf{x}], L_n[\mathbf{y}]\right) \leq \frac{1}{n} \sum_{j=1}^n\left|x_j-y_j\right|^p,
\end{equation}
where $\mathbf{x}:=\left(x^1, \ldots, x^n\right) \in \mathbb{R}^n$, and, in particular,
$$
\mathcal{W}_p^p\left(L_n[\mathbf{x}], L_n[\mathbf{0}]\right) \leq \frac{1}{n} \sum_{j=1}^n\left|x_j\right|^p,
$$
where we note that $L_n[\mathbf{0}]=\frac{1}{n} \sum_{j=1}^n \delta_0=\delta_0$.

Endow the product space $\mathbb{L}_\beta^{p, \otimes n}:=\mathbb{L}_\beta^p \times \mathbb{L}_\beta^p \times \cdots \times \mathbb{L}_\beta^p$ with the respective norm
$$
\|h\|_{\mathbb{L}_\beta^{p, \otimes n}}^p:=\sum_{1 \leq i \leq n}\left\|h^i\right\|_{\mathbb{L}_\beta^p}^p .
$$
$\mathcal{E}^{ \otimes n}$, $\mathcal{H}_\nu^{2, p, n \otimes n}$ and $\mathcal{A}^{D, \otimes n}$ are defined in a similar way.

Note that $\mathcal{S}_\beta^{p, \otimes n}$ and $\mathbb{L}_\beta^{p, \otimes n}$ are complete metric spaces. We denote by $\mathcal{S}^{p, \otimes n}:=\mathcal{S}_0^{p, \otimes n}, \mathbb{L}^{p, \otimes n}:=\mathbb{L}_0^{p, \otimes n}$. Let $\widetilde{\Phi}: \mathbb{L}_\beta^{p, \otimes n} \longrightarrow \mathbb{L}_\beta^{p, \otimes n}$ to be the mapping that associates to a process $\mathbf{Y}:=\left(Y^{1}, Y^{2}, \ldots, Y^{n}\right)$ the process $\widetilde{\Phi}\left(\mathbf{Y}\right)=\left(\widetilde{\Phi}^1\left(\mathbf{Y}\right), \widetilde{\Phi}^2\left(\mathbf{Y}\right), \ldots, \widetilde{\Phi}^n\left(\mathbf{Y}\right)\right)$ defined by the following system: for every $i=1, \ldots n$ and $t \leq T$,
\begin{equation}
\label{phi new}
\widetilde{\Phi}(\mathbf{Y})_t=\underset{\tau \in \mathcal{T}_t^n}{\operatorname{ess} \sup } \mathcal{E}_{t, \tau}^{f^i \circ \mathbf{Y}}\left[\xi^{i} \mathbf{1}_{\{\tau=T\}}+h(\tau, Y_\tau^{i},L_n\left[\mathbf{Y}_s\right])_{s=\tau} \mathbf{1}_{\{\tau<T\}}\right],
\end{equation}
where
$$
\begin{aligned}
\mathbf{f}^i \circ \mathbf{Y}: & {[0, T] \times \Omega \times \mathbb{R} \times \left(\mathbb{E}\right)^n \mapsto \mathbb{R} } \\
& \left(\mathbf{f}^i \circ \mathbf{Y}\right)(t, \omega, y, u)=f\left(t, \omega, y, u^i, L_n\left[\mathbf{Y}_t\right](\omega)\right), \quad i=1, \ldots, n .
\end{aligned}
$$
\begin{thm}
\label{thm existence and uniqueness}
Suppose that Assumption (H1)-(H5) are satisfied for any $i=1,\ldots,n$. Suppose further that $\gamma_1$ and $\gamma_2$ satisfy
$$
\gamma_1^2+\gamma_2^2 \mathbb{E}\left[ e^{p \beta A_T}\right]<\frac{1}{2},
$$
where $2\beta \geq 2 C_f+\frac{2}{\eta}$ and $\eta \leq \frac{1}{C_f^2}$.
Then the system (\ref{chaos RBSDE}) has a unique solution in $\mathbb{L}_\beta^{2, \otimes n} \otimes \mathcal{H}_\nu^{2,2, n \otimes n} \otimes \mathcal{A}^{D, \otimes n}$.

\end{thm}

\begin{proof}
\textbf{Step 1.} We first show that $\widetilde{\Phi}$ is a well-defined map from $\mathbb{L}_\beta^{2, \otimes n}$ to itself. To this end, we linearize the mapping $h$ as follows: for $i=1, \ldots, n$ and $0 \leq s \leq T$,
$$
h\left(s, Y_s^{i}, L_n\left[\mathbf{Y}_s\right]\right)=h\left(s, 0, L_n[\mathbf{0}]\right)+a_h^i(s) Y_s^{i}+b_h^i(s) \mathcal{W}_2\left(L_n\left[\mathbf{Y}_s\right], L_n[\mathbf{0}]\right),
$$
where $a_h^i(\cdot)$ and $b_h^i(\cdot)$ are adapted processes given by
\begin{equation}
\left\{\begin{array}{l}
a_h^i(s):=\frac{h\left(s, Y_s^{i}, L_n\left[\mathbf{Y}_s\right]\right)-h\left(s, 0, L_n\left[\mathbf{Y}_s\right]\right)}{Y_s^{i}} \mathbf{1}_{\left\{Y_s^{i} \neq 0\right\}}, \\
b_h^i(s):=\frac{h\left(s, 0, L_n\left[\mathbf{Y}_s\right]\right)-h\left(s, 0, L_n[\mathbf{0}]\right)}{\mathcal{W}_2\left(L_n\left[\mathbf{Y}_s\right], L_n[\mathbf{0}]\right)} \mathbf{1}_{\left\{\mathcal{W}_2\left(L_n\left[\mathbf{Y}_s\right], L_n[\mathbf{0}]\right) \neq 0\right\}} .
\end{array}\right.
\end{equation}
and which, by the Lipschitz continuity of $h$, satisfy $\left|a_h^i(\cdot)\right| \leq \gamma_1,\left|b_h^i(\cdot)\right| \leq \gamma_2$.
By Proposition 3.1 in \cite{lin2024mean} we obtain the following for any stopping time $\tau \in \mathcal{T}_t^n$:
$$
\begin{aligned}
& \left.\mid \mathcal{E}_{t, \tau}^{\mathrm{f}^i \circ \mathbf{Y}}\left[\xi^{i} 1_{\{\tau=T\}}+h\left(\tau, Y_\tau^{i}, L_n\left[\mathbf{Y}_s\right]\right)_{s=\tau}\right) 1_{\{\tau<T\}}\right]\left.\right|^2 \\
& \leq \mathbb{E}_t\left[e^{2 \beta(A_\tau-A_t)} \left|\xi^{i} 1_{\{\tau=T\}}+h\left(\tau, Y_\tau^{i}, L_n\left[\mathbf{Y}_s\right]_{s=\tau}\right) 1_{\{\tau<T\}}\right|^2 \right] \\
& \quad +\eta \mathbb{E}_t\left[\int_t^\tau e^{2 \beta(A_s-A_t)}\left|f\left(s,0,0, L_n\left[\mathbf{Y}_s\right]\right)\right|^2 d A_s\right],
\end{aligned}
$$
with $\eta, \beta>0$ such that $\eta \leq \frac{1}{C_f^2}$ and $2\beta \geq 2 C_f+\frac{3}{\eta}$. Moreover, since the integrability condition (H3)(d) (\ref{Integrability condition}) and $\left(Y^{i}\right)_{1 \leq i \leq n} \in \mathbb{L}_\beta^{2, \otimes n}$, the non-negative c\`adl\`ag process $\left(M_t^{i, \beta}\right)_{0 \leq t \leq T}$ defined by
$$
\begin{aligned}
M_t^{i, \beta}:= & \left(e^{\beta A_t}\left|h\left(t, 0, L_n[\mathbf{0}]\right)\right|+\gamma_1 e^{\beta A_t}\left|Y_t^{i}\right|+\gamma_2 e^{\beta A_t} \mathcal{W}_2\left(L_n\left[\mathbf{Y}_t\right], L_n[\mathbf{0}]\right)+e^{\beta A_T}\left|\xi^{i}\right| \mathbf{1}_{\{t=T\}}\right)^2 \\
& +\eta\left(\int_0^t\left\{e^{\beta  A_s}\left|f\left(s, 0,0, L_n[\mathbf{0}]\right)\right|+C_f e^{\beta A_s} \mathcal{W}_2\left(L_n\left[\mathbf{Y}_s\right], L_n[\mathbf{0}]\right)\right\}^2 d A_s\right)
\end{aligned}
$$
belongs to $\mathbb{L}^1$. Thus for $i=1, \ldots, n$, it follows that $\widetilde{\Phi}\left(\mathbf{Y}\right) \in \mathbb{L}_\beta^{2, \otimes n}$.

\textbf{Step 2.} We now show that $\widetilde{\Phi}$ is a contraction on the time interval $[T-h, T]$.
Fix $\mathbf{Y}=\left(Y^{1}, \ldots, Y^{n}\right), \check{\mathbf{Y}}^n=\left(\check{Y}^{1}, \ldots, \check{Y}^{n}\right) \in \mathbb{L}_\beta^{2, \otimes n},(\hat{Y}, \tilde{Y}) \in\left(\mathcal{S}_\beta^2\right)^2$, $(\hat{U}, \tilde{U}) \in\left(\mathcal{H}_\nu^{2,2}\right)^2$. By the Lipschitz continuity of $f$ and $h$, we obtain
\begin{equation}
\begin{aligned}
\label{eq lip}
& \left|f\left(s, \hat{Y}_s, \hat{U}_s, L_n\left[\mathbf{Y}_s\right]\right)-f\left(s, \tilde{Y}_s, \tilde{U}_s, L_n\left[\check{\mathbf{Y}}_s\right]\right)\right| \leq C_f\left(\left|\hat{Y}_s-\tilde{Y}_s\right|+\left\|\hat{U}_s-\tilde{U}_s\right\|_t+\mathcal{W}_2\left(L_n\left[\mathbf{Y}_s\right], L_n\left[\check{\mathbf{Y}}_s\right]\right)\right), \\
& \left|h\left(s, Y_s^{i}, L_n\left[\mathbf{Y}_s\right]\right)-h\left(s, \bar{Y}_s^{i}, L_n\left[\check{\mathbf{Y}}_s\right]\right)\right| \leq \gamma_1\left|Y_s^{i}-\bar{Y}_s^{i}\right|+\gamma_2 \mathcal{W}_2\left(L_n\left[\mathbf{Y}_s\right], L_n\left[\check{\mathbf{Y}}_s\right]\right) . \\
&
\end{aligned}
\end{equation}
By (\ref{eq wp}), we have
\begin{equation}
\label{eq wp1}
\left.\mathcal{W}_2^2\left(L_n\left[\mathbf{Y}_s\right], L_n\left[\check{\mathbf{Y}}_s\right]\right)\right) \leq \frac{1}{n} \sum_{j=1}^n\left|Y_s^{j}-\bar{Y}_s^{j}\right|^2.
\end{equation}

Then, using equation (\ref{phi new}) and Proposition 3.1 in \cite{lin2024mean} again, for any $t \leq T$ and $i=1, \ldots, n$, we have
$$
\begin{aligned}
& \left|\widetilde{\Phi}^i\left(\mathbf{Y}^n\right)_t-\widetilde{\Phi}^i\left(\check{\mathbf{Y}}^n\right)_t\right|^2 \\
& =\left|\underset{\tau \in \mathcal{T}_t^n}{\operatorname{ess} \sup } \mathcal{E}_{t, \tau}^{\mathbf{f}^i \circ \mathbf{Y}}\left[h\left(\tau, Y_\tau^{i}, L_n\left[\mathbf{Y}_s\right]_{s=\tau}\right) \mathbf{1}_{\{\tau<T\}}+\xi^{i} \mathbf{1}_{\{\tau=T\}}\right]\right. \\
& \quad -\left.\underset{\tau \in \mathcal{T}_t^n}{\operatorname{ess} \sup } \mathcal{E}_{t, \tau}^{\mathbf{f}^i \circ \check{\mathbf{Y}}}\left[h\left(\tau, \bar{Y}_\tau^{i, n}, L_n\left[\check{\mathbf{Y}}_s\right]_{s=\tau}\right) 1_{\{\tau<T\}}+\xi^{i} \mathbf{1}_{\{\tau=T\}}\right]\right|^2 \\
& \leq \underset{\tau \in \mathcal{T}_t^n}{\operatorname{ess} \sup } \mathbb{E}_t\left[\eta\left(\int_t^\tau e^{2 \beta(A_s-A_t)}\left|\left(\mathbf{f}^i \circ \mathbf{Y}\right)\left(s, \widehat{Y}_s^{i, \tau}, \widehat{U}_s^{i, \tau}\right)-\left(\mathbf{f}^i \circ \check{\mathbf{Y}}\right)\left(s, \widehat{Y}_s^{i, \tau},  \widehat{U}_s^{i, \tau}\right)\right|^2 d A_s\right)\right. \\
&\quad  \left.+e^{2 \beta(A_\tau-A_t)}\left|h\left(\tau, Y_\tau^{i}, L_n\left[\mathbf{Y}_s\right]_{s=\tau}\right)-h\left(\tau, \bar{Y}_\tau^{i}, L_n\left[\check{\mathbf{Y}}_s\right]_{s=\tau}\right)\right|^2 \right] \\
& \leq \underset{\tau \in \mathcal{T}_t^n}{\operatorname{ess} \sup } \mathbb{E}_t\left[\eta\left(\int_t^\tau e^{2 \beta(A_s-A_t)}\left|f\left(s, \widehat{Y}_s^{i, \tau}, \widehat{U}_s^{i, i, \tau}, L_n\left[\mathbf{Y}_s\right]\right)-f\left(s, \widehat{Y}_s^{i, \tau},  \widehat{U}_s^{i, i, \tau}, L_n\left[\check{\mathbf{Y}}_s\right]\right)\right|^2 d A_s\right)\right. \\
& \quad \left.+e^{2 \beta(A_\tau-A_t)}\left|h\left(\tau, Y_\tau^{i}, L_n\left[\mathbf{Y}_s\right]_{s=\tau}\right)-h\left(\tau, \bar{Y}_\tau^{i}, L_n\left[\check{\mathbf{Y}}_s\right]_{s=\tau}\right)\right|^2 \right], \\
&
\end{aligned}
$$
where $\left(\widehat{Y}^{i, \tau}, \widehat{U}^{i, \tau}\right)$ is the solution of the BSDE associated with driver $\mathbf{f}^i \circ \check{\mathbf{Y}}^n$, terminal time $\tau$ and terminal condition $h\left(\tau, \bar{Y}_\tau^{i, n}, L_n\left[\check{\mathbf{Y}}_s\right]_{s=\tau}\right) \mathbf{1}_{\{\tau<T\}}+\xi^{i} \mathbf{1}_{\{\tau=T\}}$.
Therefore, using (\ref{eq lip}) and (\ref{eq wp1}), we have, for any $t \in [T-h, T]$ and $i=1, \ldots, n$,
$$
\begin{aligned}
&e^{2 \beta A_t}\left|\widetilde{\Phi}^i(\mathbf{Y})_t-\widetilde{\Phi}^i(\check{\mathbf{Y}})_t\right|^2\\
&\leq \underset{\tau \in \mathcal{T}_t^n}{\operatorname{ess} \sup }\ \mathbb{E}_t\left[\int_t^\tau e^{2\beta A_s} \eta C_f^2   \left(\frac{1}{n} \sum_{j=1}^n \left|Y_s^{j}-\check{Y}_s^{j}\right|^2\right) d A_s+e^{2 \beta A_\tau}\left(\gamma_1\left|Y_\tau^{i}-\check{Y}_\tau^{i}\right|+\gamma_2\left\{\left(\frac{1}{n} \sum_{j=1}^n \left|Y_\tau^{j}-\check{Y}_\tau^{j}\right|^2\right)\right\}^{1 / 2}\right)^2 \right]\\
&\leq  \underset{\tau \in \mathcal{T}_t^n}{\operatorname{ess} \sup }  \ \mathbb{E}_t \left[\eta C_f^2 \sup_{s \in [T-h ,T]}\left(\frac{1}{n} \sum_{j=1}^n \left|Y_s^{j}-\check{Y}_s^{j}\right|^2\right)  \int_{T-h}^T e^{2 \beta A_s}  d A_s +e^{2 \beta A_\tau}\left\{2 \gamma_1^2\left|Y_\tau^{i}-\check{Y}_\tau^{i}\right|^2+2 \gamma_2^2 \left(\frac{1}{n} \sum_{j=1}^n \left|Y_s^{j}-\check{Y}_s^{j}\right|^2\right)_{\mid s=\tau}\right\} \right].
\end{aligned}
$$
Therefore,
$$
e^{2 \beta A_t}\left|\widetilde{\Phi}^i(\mathbf{Y})_t-\widetilde{\Phi}^i(\check{\mathbf{Y}})_t\right|^2 \leq \underset{\tau \in \mathcal{T}_t^n}{\operatorname{ess} \sup } \ \mathbb{E}_t\left[G(\tau)\right]:=V_t,
$$
in which
$$
\begin{aligned}
G(\tau)&:= \eta C_f^2 \sup_{s \in [T-h ,T]}\left(\frac{1}{n} \sum_{j=1}^n \left|Y_s^{j}-\check{Y}_s^{j}\right|^2\right) \int_{T-h}^T e^{2\beta A_s}  d A_s+e^{2 \beta A_\tau}\left( 2 \gamma_1^2\left|Y_\tau^{i}-\check{Y}_\tau^{i}\right|^2+2 \gamma_2^2 \left(\frac{1}{n} \sum_{j=1}^n \left|Y_s^{j}-\check{Y}_s^{j}\right|^2\right)_{\mid s=\tau}\right),
\end{aligned}
$$
which yields
$$
\sup _{\tau \in \mathcal{T}_{T-h}^n} \mathbb{E}\left[e^{2 \beta A_\tau}\left|\widetilde{\Phi}^i(\mathbf{Y}^n)_\tau-\widetilde{\Phi}^i(\check{\mathbf{Y}}^n)_\tau\right|^2\right] \leq \sup_{\tau \in \mathcal{T}_{T-h}^n} \mathbb{E}\left[V_\tau\right].
$$
By Lemma D.1 in \cite{karatzas1998methods}, for any $\tau \in \mathcal{T}_{T-h}^n$, there exists a sequence $\left(\tau_n\right)_n$ of stopping times in $\mathcal{T}_\tau^n$ such that
$$
V_\tau=\lim _{n \rightarrow \infty} \mathbb{E}\left[G\left(\tau_n\right) \mid \mathcal{F}_\tau\right]
$$
and so, by Fatou's Lemma, we have
$$
\mathbb{E}\left[e^{2 \beta A_\tau}\left|\widetilde{\Phi}^i(\mathbf{Y}^n)_\tau-\widetilde{\Phi}^i(\check{\mathbf{Y}}^n)_\tau\right|^2\right]\leq
\mathbb{E}\left[V_\tau\right] =  \mathbb{E}\left[ \lim _{n \rightarrow \infty} \mathbb{E}\left[G\left(\tau_n\right) \mid \mathcal{F}_\tau\right] \right]   \leq \varliminf_{n \rightarrow \infty} \mathbb{E}\left[G\left(\tau_n\right)\right] \leq \sup _{\tau \in \mathcal{T}_{T-h}^n} \mathbb{E}[G(\tau)].
$$
Therefore,
$$
\left\|\widetilde{\Phi}\left(\mathbf{Y}\right)-\widetilde{\Phi}\left(\check{\mathbf{Y}}\right)\right\|_{\mathbb{L}_\beta^{2, \otimes n}[T-h, T]}^2 \leq \alpha\left\|\mathbf{Y}-\check{\mathbf{Y}}\right\|_{\mathbb{L}_\beta^{2, \otimes n}[T-h, T]}^2 ,
$$
where $\alpha:=   \eta C_f^2 \int_{T-h}^T e^{2 \beta A_s}  d A_s+   2 \left( \gamma_1^2+\gamma_2^2 \mathbb{E}\left[ e^{2 \beta A_T}\right]\right)$. 
As $\left(\gamma_1, \gamma_2\right)$ satisfies
$$
\gamma_1^2+\gamma_2^2 \mathbb{E}\left[ e^{2 \beta A_T}\right]<2^{-1},
$$
we can choose a $\beta$ and a small enough $h$ such that $T=nh$ and 
$$
\max_{1\le i\le n} \int_{(i-1)h}^{ih} e^{2 \beta A_s}  d A_s < \frac{1}{\eta C_f^2}\left( 1 - 2 \left( \gamma_1^2+\gamma_2^2 \mathbb{E}\left[ e^{p \beta A_T}\right]\right)\right);
$$
to make $\widetilde{\Phi}$ a contraction on $\mathbb{L}_\beta^{2, \otimes n}([T-h, T])$, i.e. $\widetilde{\Phi}$ admits a unique fixed point over $[T-h, T]$.

\textbf{Step 3.} Denote
$$
\begin{aligned}
& \boldsymbol{\xi}:=\left(\xi^{1}, \ldots, \xi^{n}\right) ;  \mathbf{U}:=\left(U^{i, 1}, \ldots, U^{i, n}\right)_{i=1, \ldots, n} ; \quad \mathbf{K}:=\left(K^{1}, \ldots, K^{n}\right) .
\end{aligned}
$$
According to the classical stitching technique, we can construct the global solution  $(\mathbf{Y}, \mathbf{U}, \mathbf{K})$ of (\ref{chaos RBSDE}) backwardly over $[0, T]$, since the choice of $h$ only depends on the constants in the Assumption. The uniqueness of the solution follows by the piecewise uniqueness. We complete the proof.

\end{proof}

\subsection{Convergence result}
Let us consider $\left(\bar{Y}^i, \bar{U}^i, \bar{K}^i\right)$ independent copies of $(Y, U, K)$. More precisely,  for each $i=1, \ldots, n$, $\left(\bar{Y}^i, \bar{U}^i, \bar{K}^i\right)$, is the unique solution of the reflected MF-BSDE
\begin{equation}
\label{copy BSDE}
\left\{\begin{array}{l}
\bar{Y}_t^i=\xi^i+\int_t^T f\left(s, \bar{Y}_s^i, \bar{U}_s^i, \mathbb{P}_{\bar{Y}_s^i}\right) d A_s+\bar{K}_T^i-\bar{K}_t^i-\int_t^T \int_{E} \bar{U}_s^i(e) q^i(d s, d e), \\
\bar{Y}_t^i \geq h\left(t, \bar{Y}_t^i, \mathbb{P}_{\bar{Y}_t^i}\right), \quad \forall t \in[0, T], \\
\int_0^T\left(\bar{Y}_{t}^i-h\left(t, \bar{Y}_{t}^i, \mathbb{P}_{\bar{Y}_{t}}\right)\right) d K_t^i=0 .
\end{array}\right.
\end{equation}

In the sequel, we denote $\left(f \circ \bar{Y}^i\right)(t, y, u):=f\left(t, y, u, \mathbb{P}_{\bar{Y}_t^i}\right)$.

Now we are going to derive the following law of large numbers.
\begin{thm}[Law of Large Numbers]
\label{thm Law of Large Numbers}
Let $\bar{Y}^1, \bar{Y}^2, \ldots, \bar{Y}^n$ with terminal values $\bar{Y}_T^i=\xi^i$ be independent copies of the solution $Y$ of (\ref{reflected BSDE}). Define $\mathbf{\bar{Y}}=(\bar{Y}^1, \bar{Y}^2, \ldots, \bar{Y}^n)$, then we have
$$
\lim _{n \rightarrow \infty} \mathbb{E}\left[\sup _{0 \leq t \leq T} \mathcal{W}_2^2\left(L_n\left[\mathbf{\bar{Y}}_t\right], \mathbb{P}_{Y_t}\right)\right]=
\lim _{n \rightarrow \infty} \mathbb{E}\left[\sup _{0 \leq t \leq T} \mathcal{W}_2^2\left(L_n\left[\mathbf{\bar{Y}}_t\right], \mathbb{P}_{\bar{Y}^i_t}\right)\right]=0.
$$
\end{thm}
The proof of Theorem \ref{thm Law of Large Numbers} similar with Theorem 4.1 in \cite{Djehiche2021}.

We now provide the following convergence result for the solution $Y^{i}$ of (\ref{chaos RBSDE}).

\begin{prop}[Convergence of the $Y^{i}$ 's]
\label{prop cvg rate}
Assume that $\gamma_1$ and $\gamma_2$ satisfy
\begin{equation}
\label{cvg gamma condition}
\gamma_1^2+\gamma_2^2<\frac{1}{8}.
\end{equation}
Then, under Assumption (H1),(H2),(H3), (H4) and (H5), we have
$$
\lim _{n \rightarrow \infty} \sup _{0 \leq t \leq T} E\left[\left|Y_t^{i}-\bar{Y}_t^i\right|^2\right]=0.
$$
\end{prop}
\begin{proof}
For any $t \in[0, T]$, let $\vartheta \in \mathcal{T}_t^n$. By the estimates on BSDEs from Proposition 3.1 in \cite{lin2024mean}, we have
\begin{equation}
\label{eq esssup}
\begin{aligned}
& \left|Y_{\vartheta}^{i}-\bar{Y}_{\vartheta}^i\right|^2 \\
& \leq \underset{\tau \in \mathcal{T}_{\vartheta}^n}{\operatorname{ess} \sup }\left|\mathcal{E}_{\vartheta, \tau}^{\mathbf{f}^i \circ \mathbf{Y}}\left[h\left(\tau, Y_\tau^{i}, L_n\left[\mathbf{Y}_\tau\right]\right) \mathbf{1}_{\{\tau<T\}}+\xi^{i} \mathbf{1}_{\{\tau=T\}}\right]-\mathcal{E}_{\vartheta, \tau}^{\mathbf{f}^i \circ \bar{Y}^i}\left[h\left(\tau, \bar{Y}_\tau^i, \mathbb{P}_{Y_s \mid s=\tau}\right) \mathbf{1}_{\{\tau<T\}}+\xi^i \mathbf{1}_{\{\tau=T\}}\right]\right|^2 \\
& \leq \underset{\tau \in \mathcal{T}_{\vartheta}^n}{\operatorname{ess} \sup }\mathbb{E}_{\vartheta}\left[\eta\left(\int_{\vartheta}^\tau e^{2 \beta(A_s-A_\vartheta)}\left|\left(\mathbf{f}^i \circ \mathbf{Y}\right)\left(s, \widehat{Y}_s^{i, \tau}, \widehat{U}_s^{i, \tau}\right)-\left(\mathbf{f}^i \circ \bar{Y}^i\right)\left(s, \widehat{Y}_s^{i, \tau}, \widehat{U}_s^{i, \tau}\right)\right|^2 d A_s\right)\right. \\
& \quad \left.+\left(e^{\beta(A_\tau-A_\vartheta)}\left|h\left(\tau, Y_\tau^{i}, L_n\left[\mathbf{Y}_\tau\right]\right)-h\left(\tau, \bar{Y}_\tau^i, \mathbb{P}_{Y_s \mid s=\tau}\right)\right|\right)^2 \right] \\
& \leq \underset{\tau \in \mathcal{T}_{\vartheta}^n}{\operatorname{ess} \sup }\mathbb{E}_{\vartheta}\left[\int_{\vartheta}^\tau e^{ \beta(A_s-A_\vartheta)} \|A_T\|_\infty \eta C_f^2 \mathcal{W}_2^2\left(L_n\left[\mathbf{Y}_s\right], \mathbb{P}_{Y_s}\right) d A_s+\left(\gamma_1 e^{\beta(A_\tau-A_\vartheta)}\left|Y_\tau^{i}-\bar{Y}_\tau^i\right|\right.\right. \\
& \quad \left.\left.+\gamma_2 e^{\beta(A_\tau-A_\vartheta)} \mathcal{W}_2\left(L_n\left[\mathbf{Y}_\tau\right], \mathbb{P}_{Y_s \mid s=\tau}\right)\right)^2 \right], \\
&
\end{aligned}   
\end{equation}
where $\eta, \beta>0$ such that $\eta \leq \frac{1}{C_f^2}$ and $2\beta \geq 2 C_f+\frac{3}{\eta}$, and $\left(\widehat{Y}^{i, \tau}, \widehat{U}^{i, \tau}\right)$ is the solution of the BSDE associated with driver $\mathbf{f}^i \circ \mathbf{Y}$, terminal time $\tau$ and terminal condition $h\left(\tau, Y_\tau^{i}, L_n\left[\mathbf{Y}_s\right]_{s=\tau}\right) \mathbf{1}_{\{\tau<T\}}+$ $\xi^{i} \mathbf{1}_{\{\tau=T\}}$.
Therefore, we have
$$
e^{2 \beta \vartheta}\left|Y_{\vartheta}^{i}-\bar{Y}_{\vartheta}^i\right|^2 \leq \underset{\tau \in \mathcal{T}_{\vartheta}^n}{\operatorname{ess} \sup }  
 \mathbb{E}_{\vartheta}\left[G_{\vartheta, \tau}^{i, n} \right],
$$
where
$$
\begin{aligned}
G_{\vartheta, \tau}^{i, n}& :=\int_{\vartheta}^\tau 2 \eta\|A_T\|_\infty  C_f^2\left(\frac{1}{n} \sum_{j=1}^n e^{2 \beta A_s}\left|Y_s^{j}-\bar{Y}_s^j\right|^2\right) d A_s +2^{2}\left(\gamma_1^2+\gamma_2^2\right)\left(e^{2 \beta A_\tau}\left|Y_\tau^{i}-\bar{Y}_\tau^i\right|^2+\frac{1}{n} \sum_{j=1}^n e^{2 \beta A_\tau}\left|Y_\tau^{j}-\bar{Y}_\tau^j\right|^2\right) \\
&\quad +\left(2^{2} \gamma_2^2+2 \eta\|A_T\|_\infty  C_f^2\right) \sup _{0 \leq s \leq T} e^{2 \beta A_s} \mathcal{W}_2^2\left(L_n\left[\bar{\mathbf{Y}}_s\right], \mathbb{P}_{Y_s}\right).
\end{aligned}
$$
Setting $V_t^{n, 2}:=\frac{1}{n} \sum_{j=1}^n e^{p \beta A_t}\left|Y_t^{j}-\bar{Y}_t^j\right|^2$ and
$$
\Gamma_{n, 2}:=\left\{\left(2^{2} \gamma_2^2+2 \eta \|A_T\|_\infty C_f^2\right) \sup _{0 \leq s \leq T} e^{2 \beta A_s} \mathcal{W}_2^2\left(L_n\left[\bar{\mathbf{Y}}_s\right], \mathbb{P}_{Y_s}\right)\right\} .
$$
We obtain
$$
V_{\vartheta}^{n, 2} \leq \underset{\tau \in \mathcal{T}_{\vartheta}^n}{\operatorname{ess} \sup } \mathbb{E}_{\vartheta}\left[\int_{\vartheta}^\tau 2 \eta\|A_T\|_\infty C_f^2 V_s^{n,2} d A_s+2^3\left(\gamma_1^2+\gamma_2^2\right) V_\tau^{n, 2}+\Gamma_{n, 2}\right] .
$$
Hence, we have
$$
\mathbb{E}\left[V_{\vartheta}^{n, 2}\right] \leq 2^{3}\left(\gamma_1^2+\gamma_2^2\right) \sup _{\tau \in \mathcal{T}_{\vartheta}^n} \mathbb{E}\left[V_\tau^{n, 2}\right]+\mathbb{E}\left[\int_{\vartheta}^T 2 \eta \|A_T\|_\infty C_f^2 V_s^{n, 2} d A_s+\Gamma_{n, 2}\right] .
$$
Since for any $\vartheta \in \mathcal{T}_t^n, \mathcal{T}_{\vartheta}^n \subset \mathcal{T}_t^n$, we have
$$
\lambda \sup _{t \leq s \leq T} \mathbb{E}\left[V_s^{n,2}\right] \leq \lambda \sup _{\vartheta \in \mathcal{T}_t^n} \mathbb{E}\left[V_{\vartheta}^{n, 2}\right] \leq \mathbb{E}\left[\int_t^T 2 \eta \|A_T\|_\infty C_f^2 V_s^{n, 2} d A_s+\Gamma_{n, 2}\right] .
$$
In particular,
$$
\lambda \mathbb{E}\left[V_t^{n, 2}\right] \leq \mathbb{E}\left[\int_t^T 2\eta C_f^2 V_s^{n, 2} d A_s+\Gamma_{n, 2}\right],
$$
where $\lambda:=1-2^{3}\left(\gamma_1^2+\gamma_2^2\right)>0$ by the assumption (\ref{cvg gamma condition}). Now, by the backward Gronwall inequality, we obtain
$$
\mathbb{E}\left[V_t^{n, 2}\right] \leq \frac{e^{K_p}}{\lambda} \mathbb{E}\left[\Gamma_{n, 2}\right],
$$
where $K_p:=\frac{1}{\lambda} 2 \eta \|A_T\|_\infty C_f^2 $. Hence, since the choice of $t \in[0, T]$ is arbitrary, we must have
$$
\sup _{0 \leq t \leq T} \mathbb{E}\left[V_t^{n, 2}\right] \leq \frac{e^{K_p}}{\lambda} \mathbb{E}\left[\Gamma_{n, 2}\right].
$$
Indeed, we have $\mathbb{E}\left[V_t^{n, 2}\right]=\mathbb{E}\left[e^{2 \beta A_t}\left|Y_t^{i}-\bar{Y}_t^i\right|^2\right]$. Thus,
\begin{equation}
\label{eq EsupY leq law}
\sup _{0 \leq t \leq T} \mathbb{E}\left[e^{2 \beta A_t}\left|Y_t^{i}-\bar{Y}_t^i\right|^2\right] \leq \frac{e^{K_p}}{\lambda} \mathbb{E}\left[\Gamma_{n, 2}\right] \rightarrow 0
\end{equation}
as $n \rightarrow \infty$, in view of Theorem \ref{thm Law of Large Numbers}, as required.
\end{proof}

\section{Well-posedness and convergence result of MF-RBSDEs driven by a Poisson process}
\label{sec cvg rate}
In this section, we consider a special mean-field reflected BSDE with jumps under the Poisson process framework:
\begin{equation}
\label{eq reflected BSDE Poisson}
\left\{\begin{array}{l}
Y_t=\xi+\int_t^T f(s, Y_s, U_s, \mathbb{P}_{Y_s}) d s +\int_t^T d K_s-\int_t^T \int_E U_s(e) \tilde{\mu}(d s, d e), \quad \forall t \in[0, T], \quad \mathbb{P} \text {-a.s., } \\
Y_t \geq h(t, Y_t, \mathbb{P}_{Y_t}), \quad \forall t \in[0, T],\\
\int_0^T\left(Y_{t}-h\left(t, Y_{t}, \mathbb{P}_{Y_{t}}\right)\right) d K_t=0.
\end{array}\right.
\end{equation}
where we define a random measure $\mu(\omega,dt,de)$ on $\mathbb{R}_{+} \times E$, where $E \triangleq \mathbb{R}^{\ell} \backslash\{0\}$ is equipped with its Borel field $\mathcal{E}$, with compensator $\lambda(\omega, d t, d e)$.
During the paper, we assume that $\lambda$ is absolutely continuous with respect to the Lebesgue measure $d t$, i.e. $\lambda(\omega, d t, d e)=\nu_t(\omega, d e) d t$. Finally, we denote $\tilde{\mu}$ the compensated jump measure
$$
\widetilde{\mu}(\omega, d e, d t)=\mu(\omega, d e, d t)-\nu_t(\omega, d e) d t .
$$

Similar with Section \ref{sec chaos A}, we study a family of weakly interacting process 
\begin{equation}
\label{chaos RBSDE Poisson}
\left\{\begin{aligned}
& Y_t^{i}= \xi^{i}+\int_t^T f\left(s, Y_s^{i}, U_s^{i, i}, L_n\left[\mathbf{Y}_s\right]\right) d s+K_T^{i}-K_t^{i} -\int_t^T \int_{E} \sum_{j=1}^n U_s^{i, j}(e) \tilde{\mu}^j(d s, d e), \quad \forall t \in[0, T], \quad \mathbb{P} \text {-a.s., }\\
& Y_t^{i} \geq h\left(t, Y_t^{i, n}, L_n\left[\mathbf{Y}_t\right]\right), \quad \forall t \in[0, T],\\
& \int_0^T\left(Y_{t}^{i}-h\left(t, Y_{t}^{i}, L_n\left[\mathbf{Y}_{t}\right]\right)\right) d K_t^{i}=0.
\end{aligned}\right.   
\end{equation}
where ${\xi^i}_{1 \leq i \leq N}$, ${f^i}_{1 \leq i \leq N}$ and ${\tilde{\mu}^i}_{1 \leq i \leq N}$ are independent copies of $\xi, f$ and $\tilde{\mu}$. The augmented natural filtration of the family of $\{\tilde{\mu}^i\}_{1\leq i \leq N}$ is denoted by $\mathscr{F}_t^{(N)}$. Denote by $\mathbb{F}^N:=\left\{\mathcal{F}_t^N\right\}_{t \in[0, T]}$ the completion of the filtration generated by $\left\{\tilde{\mu}^i\right\}_{1 \leq i \leq N}$. Let $\mathcal{T}_t^N$ be the set of $\mathbb{F}^N$ stopping times with values in $[t, T]$. 

\subsection{Well-posedness result}
In order to show the convergence rate of the solution of this system to the solution of the mean field reflected BSDE (\ref{eq reflected BSDE Poisson}), we need some stricter assumptions.

\hspace*{\fill}\\
\hspace*{\fill}\\
\noindent(\textbf{H2'}) For every $\omega \in \Omega,\ t \in[0, T],\ r \in \mathbb{R}$, $\mu \in \mathcal{P}_1(\mathbb{R})$ the mapping
$
f(\omega, t, r, \cdot, \mu):L^2(E,\mathcal{B}(E),\nu_t(\omega,dy)) \rightarrow \mathbb{R}
$
satisfies:
for every $U \in {H_\nu^{2,2}}$,
$$
(\omega, t, r, \mu) \mapsto f\left(\omega, t, r, U_t(\omega, \cdot), \mu \right)
$$
is Prog $\otimes \mathscr{B}(\mathbb{R})$-measurable.

\hspace*{\fill}\\
\hspace*{\fill}\\

\noindent(\textbf{H3'})

\textbf{(a) (Continuity condition)} 
For every $\omega \in \Omega, t \in[0, T], y \in \mathbb{R}$, $u\in L^2(E,\mathcal{B}(E),\nu_t(\omega,dy))$, $\mu \in \mathcal{P}_1(\mathbb{R})$,  $(y, u, \mu) \longrightarrow f(t, y, u, \mu)$ is continuous. 
\hspace*{\fill}\\
\hspace*{\fill}\\

\textbf{(b) (Lipschitz condition)}  
There exists $\lambda\geq 0$, such that for every $\omega \in \Omega,\ t \in[0, T],\ y_1, y_2 \in \mathbb{R}$,\ $u\in L^2(E,\mathcal{B}(E),\nu_t(\omega,dy))$, $\mu_1, \mu_2 \in \mathcal{P}_1(\mathbb{R})$, we have
$$
\begin{aligned}
& \left|f(\omega, t, y_1, u_1, \mu_1)-f\left(\omega, t, y_2, u_2, \mu_2\right)\right| \leq \lambda\left(\left|y_1-y_2\right|+|u_1-u_2|_\nu+ \mathcal{W}_1\left(\mu_1, \mu_2\right)\right).
\end{aligned}
$$
\hspace*{\fill}\\

\textbf{(c) (Growth condition)}
For all $t \in[0, T], \ (y,u) \in \mathbb{R} \times  L^2(E,\mathcal{B}(E),\nu_t(\omega,dy)), \mu \in \mathcal{P}_1(\mathbb{R})$: $\ \mathbb{P}$-a.s, there exists $\lambda>0$ such that,
$$
-\frac{1}{\lambda} j_{\lambda}(t,- u)-\alpha_t-\beta\left(|y|+ \mathcal{W}_1\left(\mu, \delta_0\right)\right) \leq f(t, y, u, \mu) \leq \frac{1}{\lambda} j_{\lambda}(t, u)+\alpha_t+\beta\left(|y|+\mathcal{W}_1\left(\mu, \delta_0\right)\right),
$$
where $\{\alpha_t\}_{0 \leq t \leq T}$ is  a progressively measurable nonnegative stochastic process with $\|\alpha\|_{\mathcal{S}^{\infty}}<\infty$. 
\hspace*{\fill}\\
\hspace*{\fill}\\

\textbf{(d) (Bounded condition)}
For each $t \in[0, T]$, $f(t, 0,0, \delta)$ is bounded by some constant $M$, $\mathbb{P}$-a.s.
\hspace*{\fill}\\
\hspace*{\fill}\\

\textbf{(e) ($A_{\gamma}$-condition)}  
For all $t \in[0, T], M>0,\  y \in \mathbb{R},\   u_1, u_2 \in$ $\mathbb{L}^2\left(E,\mathcal{B}(E),\nu_t(\omega,dy)\right),\ \mu \in \mathcal{P}_1(\mathbb{R}),$ with $|y|,\ \|u_1\|_{\mathcal{J}^\infty},\ \left\|u_2\right\|_{\mathcal{J}^\infty} \leq M$, there exists a $\mathcal{P} \otimes \mathcal{E}$-measurable process $\gamma^{y, u_1, u_2}$ satisfying $d t \otimes d \mathbb{P}$-a.e.
$$
f(t, y, u_1,\mu)-f\left(t, y, u_2,\mu\right) \leq \int_E \gamma_t^{y, u_1,u_2}(x)\left[u_1(x)-u_2(x)\right] \nu(d x)
$$
and $C_M^1(1 \wedge|x|) \leq \gamma_t^{y,  u_1, u_2}(x) \leq C_M^2(1 \wedge|x|)$ with two constants $C_M^1, C_M^2$. Here, $C_M^1>-1$ and $C_M^2>0$ depend on $M$. (Hereafter, we frequently omit the superscripts to lighten the notation.)

\hspace*{\fill}\\
\hspace*{\fill}\\

\noindent(\textbf{H4'})
$h$ is a mapping from $[0, T] \times \Omega \times \mathbb{R} \times \mathcal{P}_1(\mathbb{R})$ into $\mathbb{R}$ such that
\hspace*{\fill}\\
\hspace*{\fill}\\

\textbf{(a) (Continuity condition)} 
For all $(y, \mu) \in \mathbb{R} \times \mathcal{P}_1(\mathbb{R}),h(\cdot, y, \mu)$ is a continuous process.
\hspace*{\fill}\\
\hspace*{\fill}\\

\textbf{(b) (Lipschitz condition)}
$h$ is Lipschitz w.r.t. $(y, \mu)$ uniformly in $(t, \omega)$, i.e. there exists two positive constants $\gamma_1$ and $\gamma_2$ such that $\mathbb{P}$-a.s. for all $t \in[0, T]$,
$$
\left|h\left(t, y_1, \mu_1\right)-h\left(t, y_2, \mu_2\right)\right| \leq \gamma_1\left|y_1-y_2\right|+\gamma_2 \mathcal{W}_1\left(\mu_1, \mu_2\right)
$$
for any $y_1, y_2 \in \mathbb{R}$ and $\mu_1, \mu_2 \in \mathcal{P}_1(\mathbb{R})$.
\hspace*{\fill}\\
\hspace*{\fill}\\

\textbf{(c)}
The final condition $\xi: \Omega \rightarrow \mathbb{R}$ is $\mathscr{F}_T$-measurable and essentially bounded, i.e., $\|\xi\|_{\infty}<\infty$. Besides, $\xi \geq h\left(T, \xi, \mathbb{P}_{\xi}\right)$.
\hspace*{\fill}\\
\hspace*{\fill}\\

\textbf{(d) (Bounded condition)}
For each $t \in[0, T]$, $h(t, 0, \delta)$ is bounded by some constant $M$, $\mathbb{P}$-a.s.
\hspace*{\fill}\\

We first give a well-posedness result of the mean field reflected BSDE (\ref{eq reflected BSDE Poisson}).

\begin{thm}
\label{thm bdd existence}
Assume (H1), (H2'), (H3') and  (H4') hold. Then the MF-BSDE (\ref{eq reflected BSDE Poisson}) has a unique solution $(Y, U, K)$ in the space $ \mathcal{S}^{\infty} \times \mathcal{J}^\infty \times \mathcal{A}^D$. 
\end{thm}

In order to prove Theorem \ref{thm bdd existence}, we need to recall some proposition to analyze the solution map.

\begin{prop}
\label{prop standard}
Assume that:

(i)The map $(\omega, t) \mapsto f(\omega, t, \cdot)$ is $\mathbb{F}$-progressively measurable. There exist a positive $\mathbb{F}$-progressively measurable process $\left(\alpha_t, t \in[0, T]\right)$ such that $-\alpha_t-\frac{1}{\lambda} j_{\lambda}(t,- u)  \leq f(t,u) \leq \alpha_t + \frac{1}{\lambda} j_{\lambda}(t, u) $ $d t \otimes d \mathbb{P}$-a.e. $(\omega, t) \in \Omega \times[0, T]$.

(ii) $|\xi|,\left(\alpha_t, t \in[0, T]\right)$ are essentially bounded, i.e., $\|\xi\|_{\infty},\|\alpha\|_{\mathcal{S}^{\infty}}<\infty$.

(iii) There exists $\lambda\geq 0$, such that for every $\omega \in \Omega,\ t \in[0, T],\ y_1, y_2 \in \mathbb{R}$,\ $u_1, u_2\in L^2(\mathcal{B}(E),\nu_t;\mathbb{R})$, we have
$$
\begin{aligned}
& \left|f(\omega, t, y_1, u_1)-f\left(\omega, t, y_2, u_2\right)\right| \leq \lambda\left(\left|y_1-y_2\right|+ |u_1-u_2|_\nu\right).
\end{aligned}
$$

(iv) ( $A_{\gamma}$-condition). For all $t \in[0, T], M>0, y \in \mathbb{R},  u_1, u_2 \in$ $\mathbb{L}^2\left(\mathcal{B}(E),\nu_t; \mathbb{R}\right)$ with $|y|,\|u_1\|_{\mathcal{J}^\infty},\left\|u_2\right\|_{\mathcal{J}^\infty} \leq M$, there exists a $\mathcal{P} \otimes \mathcal{E}$-measurable process $\gamma^{y, u_1, u_2}$ satisfying $d t \otimes d \mathbb{P}$-a.e.
$$
f(t, y, u_1)-f\left(t, y, u_2\right) \leq \int_E \gamma_t^{y, u_1,u_2}(x)\left[u_1(x)-u_2(x)\right] \nu(d x)
$$
and $C_M^1(1 \wedge|x|) \leq \gamma_t^{y, u_1, u_2}(x) \leq C_M^2(1 \wedge|x|)$ with two constants $C_M^1, C_M^2$. Here, $C_M^1>-1$ and $C_M^2>0$ depend on $M$. (Hereafter, we frequently omit the superscripts to lighten the notation.)

Then BSDE
\begin{equation}
\label{eq standard}
\bar{Y}_t= \xi + \int_t^T f(s,\bar{U}_s(e)) ds -\int_t^T \int_E \bar{U}_s(e) \tilde{\mu}(ds,de)
\end{equation}
has a unique solution $(\bar{Y},\bar{U})$ in the space $\mathcal{S}^{\infty} \times \mathcal{J}^{\infty}$.
\end{prop}

\begin{proof}
The well-posedness can be deduced by Theorem 4.1 in \cite{fujii2018quadratic}.
\end{proof}

\begin{thm}
\label{thm fixed mean reflected}
Assume that the same conditions as Theorem \ref{thm bdd existence} hold. $P \in \mathcal{S}^{\infty}$. Then the following BSDE 
\begin{equation}
\label{eq fixed}
\left\{\begin{array}{l}
Y_t=\xi+\int_t^T f(s, P_s, U_s, \mathbb{P}_{P_s}) d s-\int_t^T\int_E U_s(e) \tilde{\mu}(ds,de)+K_T-\tilde{K}_t , \quad \forall t \in[0, T] , \quad \mathbb{P} \text {-a.s., }\\
Y_t \geq h(t, P_t, \mathbb{P}_{P_t}), \quad 0 \leq t \leq T,\\
\int_0^T\left(Y_{t }-h\left(t , P_{t }, \mathbb{P}_{P_{t }}\right)\right) d K_t=0.
\end{array}\right. 
\end{equation}
has a unique solution $(Y, U, K)$ in the space $ \mathcal{S}^{\infty} \times \mathcal{J}^\infty \times \mathcal{A}^D$.  
\end{thm}

\begin{proof}
Obviously, the generator $f(s, P_s, U_s(e), \mathbb{P}_{P_s})$, the terminal $\xi$ and the loss function $h$ satisfies the conditions of Theorem 4.1 in \cite{foresta2021optimal}, thus the BSDE (\ref{eq fixed}) admits a unique deterministic flat solution $(\tilde{Y}, \tilde{U}, \tilde{K}) \in L^{2, \beta}(s) \times L^{2, \beta}(p) \times \mathcal{A}_D$.

Moreover, for each $t \in[0,T]$, in view of the Theorem 3.3 in \cite{quenez2014reflected}, we have
$$
Y_t=\underset{\tau \in \mathcal{T}_t}{\operatorname{ess} \sup } \mathcal{E}_{t, \tau}^{f\circ P}\left[\xi \mathbf{1}_{\{\tau=T\}}+h(\tau, P_\tau,\mathbb{P}_{P_s\mid s = \tau}) \mathbf{1}_{\{\tau<T\}}\right],
$$
where $(f \circ P)(t,y,u)=f\left(t,P, u, \mathbb{P}_{P_t}\right)$ and $\mathcal{E}_{t, \tau}^{f\circ P}\left[\xi \mathbf{1}_{\{\tau=T\}}+h(\tau, P_\tau,\mathbb{P}_{P_s\mid s = \tau}) \mathbf{1}_{\{\tau<T\}}\right]:=y_t^\tau \in \mathcal{S}^{\infty}$ is the solution to the following standard BSDE:
\begin{equation}
y_t^\tau=\xi \mathbf{1}_{\{\tau=T\}}+h_\tau(P_\tau,\mathbb{P}_{P_s\mid s = \tau}) \mathbf{1}_{\{\tau<T\}}+\int_t^\tau f\left(s,P_s, u_s^\tau, \mathbb{P}_{P_s})\right) d s-\int_t^\tau \int_E u_s^\tau(e) \tilde{\mu}(ds, de). 
\end{equation}
Consequently, $Y_t \in  \mathcal{S}^{\infty}$ and from Corollary 1 in \cite{morlais2009quadratic}, we know that $\|U\|_{\mathcal{J}^\infty} \leq 2\|Y\|_{\mathcal{S}^\infty} < \infty$ ,
Thus $(Y, U, K) \in \mathcal{S}^{\infty} \times J^\infty \times \mathcal{A}^D$ is the unique  solution to the BSDE (\ref{eq fixed}).
\end{proof}

We are now ready to complete the proof Theorem \ref{thm bdd existence}.
\begin{proof}[Proof of Theorem \ref{thm bdd existence}]
Let $P^i \in \mathcal{S}^{\infty}, i=1,2$. It follows from the $g$-expectation representation lemma that
$$
\Gamma\left(P^i\right)_t:=\operatorname{essup}_{\tau \in \mathcal{T}_t} y_t^{i, \tau}, \quad \forall t \in[0, T]
$$
in which $y_t^{i,\tau}$ is the solution to the following BSDE:
\begin{equation}
y_t^{i,\tau}=\xi \mathbf{1}_{\{\tau=T\}}+h_\tau(P^i_\tau,\mathbb{P}_{P^i_s\mid s = \tau}) \mathbf{1}_{\{\tau<T\}}+\int_t^\tau f\left(s, P_s^i, u_s^{i,\tau}, \mathbb{P}_{P_s^i})\right) d s-\int_t^\tau \int_E u_s^{i,\tau}(e) \tilde{\mu}(ds, de). 
\end{equation}
For each $t \in[0, T]$, denote by
$$
f^{P^i}\left(s, u_s^{i,\tau}\right) = f\left(t, P_t^i, u_t^{i,\tau}, \mathbb{P}_{P_t^i})\right).
$$
Then, the pair of processes $\left(y^{1, \tau}-y^{2, \tau}, u^{1, \tau}-u^{2, \tau}\right)$ solves the following BSDE in $t \in [T-h,T]$:
$$
\begin{aligned}
y_t^{1, \tau}-y_t^{2, \tau}&=h\left(\tau, P_\tau^1,\mathbf{P}_{P_s^1 \mid {s=\tau}}\right) \mathbf{1}_{\{\tau<T\}}-h\left(\tau, P_\tau^2,\mathbf{P}_{P_s^2 \mid {s=\tau}}\right) \mathbf{1}_{\{\tau<T\}} + \int_t^\tau \left(f^{P^1}\left(s, u_s^{1,\tau}\right) -f^{P^2}\left(s, u_s^{2,\tau}\right)\right) ds \\
&\quad +\int_t^\tau \int_E \left(u_s^{1,\tau}(e)-u_s^{2,\tau}(e)\right) \tilde{\mu}(ds,de) \mathbf{1}_{[0, \tau]}(s) d s\\
& = h\left(\tau, P_\tau^1,\mathbf{P}_{P_s^1 \mid {s=\tau}}\right) \mathbf{1}_{\{\tau<T\}}-h\left(\tau, P_\tau^2,\mathbf{P}_{P_s^2 \mid {s=\tau}}\right) \mathbf{1}_{\{\tau<T\}} \\
&\quad + \int_t^\tau \left[f^{P^1}\left(s, u_s^{1,\tau}\right) -f^{P^2}\left(s, u_s^{1,\tau}\right) + f^{P^2}\left(s, u_s^{1,\tau}\right)-f^{P^2}\left(s, u_s^{2,\tau}\right)\right] ds \\
&\quad +\int_t^\tau \int_E \left(u_s^{1,\tau}(e)-u_s^{2,\tau}(e)\right) \tilde{\mu}(ds,de) \mathbf{1}_{[0, \tau]}(s) d s\\
&\leq \mathbb{E}_t^{\tilde{\mu}^\gamma}\left[h\left(\tau, P_\tau^1,\mathbf{P}_{P_s^1 \mid {s=\tau}}\right) \mathbf{1}_{\{\tau<T\}}-h\left(\tau, P_\tau^2,\mathbf{P}_{P_s^2 \mid {s=\tau}}\right) \mathbf{1}_{\{\tau<T\}} + \int_t^\tau \lambda \left(\left|P_s^1 -P_s^2\right| + \mathcal{W}_1\left(\mathbb{P}_{P_s^1},\mathbb{P}_{P_s^2}\right) \right)ds\right]\\
&\leq \mathbb{E}_t^{\tilde{\mu}^\gamma}\left[\left(\gamma_1+\lambda h\right)\sup_{T-h,T} \left|P_s^1 -P_s^2\right| + \left(\gamma_2+\lambda h\right)\sup_{T-h,T} \mathbb{E}\left[\left|P_s^1 -P_s^2\right|\right]\right],
\end{aligned}
$$
where $\tilde{\mu}^\gamma := \tilde{\mu} -<\tilde{\mu}, \tilde{\gamma}\cdot \tilde{\mu}> $, $\tilde{\gamma} = \gamma(u,u')$.
Thus, we can deduce that
$$\left\|\Gamma(P^1)-\Gamma(P^2) \right\|_{\mathcal{S}^\infty} \leq \left(\gamma_1 +\gamma_2 +2 \lambda h  \right)\left\|P^1-P^2 \right\|_{\mathcal{S}^\infty}.$$

We can then find a small enough constant $h$ depending only on $\lambda, \gamma_1$ and $\gamma_2$ such that $\gamma_1+\gamma_2+2 \lambda h<1$. It is now obvious that $\Gamma$ defines a contraction map on the time interval $[T-h, T]$.
The uniqueness of the global solution on $[0, T]$ is inherited from the uniqueness of the local solution on each small time interval. It suffices to prove the existence.

We already know that there exists a constant $\delta>0$ depending only on $ \lambda, \gamma_1$ and $\gamma_2$, such that the mean-field reflected BSDE (\ref{eq reflected BSDE Poisson}) admits a unique solution
$$
\left(Y^1, U^1, K^1\right) \in \mathcal{S}_{[T-\delta, T]}^\infty \times \mathcal{J}_{[T-\delta, T]}^\infty\ \times \mathcal{A}_{[T-\delta, T]}^D
$$
on the time interval $[T-\delta, T]$. Next, taking $T-\delta$ as the terminal time, then the mean-field reflected BSDE (\ref{eq reflected BSDE Poisson}) admits a unique solution
$$
\left(Y^2, U^2, K^2\right) \in \mathcal{S}_{[T-2 \delta, T-\delta]}^\infty \times \mathcal{J}_{[T-2 \delta, T-\delta]}^\infty \times \mathcal{A}_{[T-2 \delta, T-\delta]}^D
$$
on the time interval $[T-2 \delta, T-\delta]$. Denote by
$$
\begin{aligned}
& Y_t=\sum_{i=1}^2 Y_t^i \mathbf{1}_{[T-i \delta, T-(i-1) \delta)}+Y_T^1 \mathbf{1}_{\{T\}}, U_t=\sum_{i=1}^2 U_t^i \mathbf{1}_{[T-i \delta, T-(i-1) \delta)}+U_T^1 \mathbf{1}_{\{T\}}, \\
& K_t=K_t^2 \mathbf{1}_{[T-i \delta, T-(i-1) \delta)}+\left(K_{T-\delta}^2+K_t^1\right) \mathbf{1}_{[T-\delta, T]} .
\end{aligned}
$$

It is easy to check that $(Y, U, K) \in \mathcal{S}_{[T-2 \delta, T-\delta]}^\infty \times \mathcal{J}_{[T-2 \delta, T-\delta]}^\infty \times \mathcal{A}_{[T-2 \delta, T-\delta]}^D$ is a solution to the mean-field reflected BSDE (\ref{eq reflected BSDE Poisson}). Repeating this procedure, we obtain a global solution $(Y, U, K) \in \mathcal{S}^\infty \times \mathcal{J}^\infty \times \mathcal{A}^D$. The proof of the theorem is complete.
\end{proof}

\subsection{Convergence result}
The following proposition is crucial to prove convergence result.
\begin{prop}
\label{prop regularity of mean-field}
Let $p \geq 2$ and assume that the same conditions as Theorem \ref{thm bdd existence} hold. Let $0 \leq t_1 \leq t_2 \leq T$ be such that $t_2 - t_1 \leq 1$. There exists a constant $C$ depending on $p, T, \kappa$, and $h$ such that the following hold:\\
(i) $\forall t_1 \leq s \leq t \leq t_2, \quad \mathbb{E}\left[\left|Y_t-Y_s\right|^p\right] \leq C|t-s|$.\\
(ii) $\forall t_1 \leq r<s<t \leq t_2, \quad \mathbb{E}\left[\left|Y_s-Y_r\right|^p\left|Y_t-Y_s\right|^p\right] \leq C|t-r|^2$.\\
$Y$ is the first part of the solution of BSDE (\ref{eq reflected BSDE Poisson}).
\end{prop}

\begin{proof}
(i). According to the representation theorem,
\begin{equation}
\label{eq Y and y}
\begin{aligned}
\mathbb{E}\left[\left|Y_t-Y_s\right|^p\right] &= \mathbb{E}\left[\left|\underset{\tau \geq t}{\operatorname{ess} \sup } y^{\tau}_t-\underset{\tau \geq s}{\operatorname{ess} \sup } y^{\tau}_s\right|^p\right] \\
&\leq  \mathbb{E}\left[ \underset{\tau \geq t}{\operatorname{ess} \sup }\left|y^{\tau}_t-y^{\tau}_s  \right|^p\right],
\end{aligned}
\end{equation}
where $y^{\tau}_t$ is the solution of BSDE
$$
y_t^\tau=\xi \mathbf{1}_{\{\tau=T\}}+h_\tau(P_\tau,\mathbb{P}_{Y_s\mid s = \tau}) \mathbf{1}_{\{\tau<T\}}+\int_t^\tau f\left(s,Y_s, u_s^\tau, \mathbb{P}_{Y_s})\right) d s-\int_t^\tau \int_E u_s^\tau(e) \tilde{\mu}(ds, de).
$$
Then, by H\"older's inequality
$$
\begin{aligned}
\left|y_t^\tau-y_s^\tau\right|^p &=\left(\left|\int_s^t f\left(u,Y_u,u_u^\tau, \mathbb{P}_{Y_u} \right)du\right| +\left|\int_s^t \int_E u_u^\tau(e) \tilde{\mu}(du,de)\right|\right)^p\\
&\leq 2^{p-1}\left[\left(\int_s^t \left|f\left(u,Y_u,u_u^\tau, \mathbb{P}_{Y_u} \right)\right|du\right)^p +\left|\int_s^t \int_E u_u^\tau(e) \tilde{\mu}(du,de)\right|^p\right]\\
&\leq 2^{p-1}\left[(t-1)^{p-1}\int_s^t \left|f\left(u,Y_u,u_u^\tau, \mathbb{P}_{Y_u} \right)\right|^p du +\left|\int_s^t \int_E u_u^\tau(e) \tilde{\mu}(du,de)\right|^p\right]
\end{aligned}
$$
Thus, we can deduce that
$$
\begin{aligned}
\mathbb{E}\left[\left|y_t^\tau-y_s^\tau\right|^p\right]&\leq 2^{p-1}\left\{\mathbb{E}\left[(t-s)^{p-1}\int_s^t \left|f\left(u,Y_u,u_u^\tau, \mathbb{P}_{Y_u} \right)\right|^p du \right]\right.\\
&\quad \left.+\mathbb{E}\left[\left(\int_s^t\int_E\left|u_u^\tau(e)\right|^2\nu(de)du \right)^{\frac{p}{2}}\right]+\mathbb{E}\left[\int_s^t \int_E|u_u^\tau(e)|^p \nu(de)d u \right]\right\}\\
&\leq2^{p-1}\left\{\mathbb{E}\left[(t-s)^{p-1}\int_s^t \left| \alpha_t +\beta |Y_u| + \beta \mathcal{W}_1\left(\mathbb{P}_{Y_u},\delta_0\right)+\frac{1}{\lambda}j_{\lambda}(u, u^\tau_s(e)) \right|^p du \right]\right.\\
&\quad \left.+\mathbb{E}\left[\left(\int_s^t\int_E\left|u_u^\tau(e)\right|^2\nu(de)du \right)^{\frac{p}{2}}\right]+\mathbb{E}\left[\int_s^t \int_E|u_u^\tau(e)|^p \nu(de)d u \right]\right\}\\
&\leq 2^{p-1}\left\{(t-s)^{p-1} C_1 \left(\mathbb{E}\left[\sup_{0\leq t\leq T} |Y_t|^p\right]+1\right)\right.\\
&\quad \left.+ (t-s)^{p-1} C_2 \mathbb{E}\left[\left(\int_s^t\int_E e^{u^\tau_u(e)} -1 - {u^\tau_u(e)} \nu(de)du\right)^p\right]\right. \\
&\quad \left.+C_3 |t-s|^{p / 2} \mathbb{E}\left[\|u_s^\tau(e)\|^p_{\mathcal{J}^\infty}\left(\int_E \nu(de)\right)^{p / 2}\right]+C_4 |t-s| \mathbb{E}\left[\|u_s^\tau(e)\|^p_{\mathcal{J}^\infty} \int_E \nu(de) \right]\right\}.\\ 
\end{aligned}
$$
By (\ref{eq Y and y}) we can also obtain
$$
\begin{aligned}
\mathbb{E}\left[\left|Y_t-Y_s\right|^p\right] &\leq 2^{p-1}\left\{(t-s)^{p-1} C_1 \left(\mathbb{E}\left[\sup_{0\leq t\leq T} |Y_t|^p\right]+1\right)\right.\\
&\quad \left.+ (t-s)^{p-1} C_2 \mathbb{E}\left[\left(\int_s^t\int_E e^{u^\tau_u(e)} -1 - {u^\tau_u(e)} \nu(de)du\right)^p\right]\right. \\
&\quad \left.+C_3 |t-s|^{p / 2} \mathbb{E}\left[\|u_s^\tau(e)\|^p_{\mathcal{J}^\infty}\left(\int_E \nu(de)\right)^{p / 2}\right]+C_4 |t-s| \mathbb{E}\left[\|u_s^\tau(e)\|^p_{\mathcal{J}^\infty} \int_E \nu(de) \right]\right\}.  
\end{aligned}
$$
Finally, in view of Theorem \ref{thm bdd existence}, we conclude that there exists a constant $C$, such that
$$
\forall 0 \leq s \leq t \leq T, \quad \mathbb{E}\left[\left|Y_t-Y_s\right|^p\right] \leq C|t-s| .
$$

(ii). According to the representation theorem, 
$$
\begin{aligned}
\mathbb{E}\left[\left|Y_s-Y_r\right|^p\left|Y_t-Y_s\right|^p\right] &= \mathbb{E}\left[\left|\underset{\tau \geq s}{\operatorname{ess} \sup } y^{\tau}_t-\underset{\tau \geq r}{\operatorname{ess} \sup } y^{\tau}_s\right|^p  \left|\underset{\tau \geq t}{\operatorname{ess} \sup } y^{\tau}_t-\underset{\tau \geq s}{\operatorname{ess} \sup } y^{\tau}_s\right|^p\right] \\
& \leq \mathbb{E}\left[ \underset{\tau \geq s}{\operatorname{ess} \sup }\left|y^{\tau}_t-y^{\tau}_s  \right|^p \underset{\tau \geq r}{\operatorname{ess} \sup }\left|y^{\tau}_s-y^{\tau}_r  \right|^p\right]\\
&\leq \mathbb{E}\left[ \underset{\tau \geq s}{\operatorname{ess} \sup }\left|y^{\tau}_t-y^{\tau}_s  \right|^p \underset{\tau \geq r}{\operatorname{ess} \sup } \left\{\mathbb{E}_r\left[\left| \int_r^s f(u,Y_u,u_u^\tau,\mathbb{P}_{Y_u})du\right|^p\right]+ \mathbb{E}_r\left[\left| \int_r^s \int_E u_u^\tau(e)\tilde{\mu}(du,de)\right|^p\right]  \right\}\right]\\
&\leq \mathbb{E}\left[ \underset{\tau \geq s}{\operatorname{ess} \sup }\left|y^{\tau}_t-y^{\tau}_s  \right|^p \left\{\mathbb{E}_r\left[\underset{\tau \geq r}{\operatorname{ess} \sup } \left\| u_u^\tau(e)\right\|^p_{\mathcal{J}^\infty}\int_E \nu(de)\right]|s-r|  \right\}\right]\\
&\leq C|t-s||s-r|\\
&\leq C|t-r|^2.
\end{aligned}
$$
\end{proof}

Let us consider $\left(\bar{Y}^i, \bar{U}^i, \bar{K}^i\right)$ independent copies of $(Y, U, K)$ which are the solution of MF-BSDE (\ref{eq reflected BSDE Poisson}). More precisely,  for each $i=1, \ldots, n$, $\left(\bar{Y}^i, \bar{U}^i, \bar{K}^i\right)$, is the unique solution of the reflected MF-BSDE
\begin{equation}
\label{copy BSDE ds}
\left\{\begin{array}{l}
\bar{Y}_t^i=\xi^i+\int_t^T f\left(s, \bar{Y}_s^i, \bar{U}_s^i, \mathbb{P}_{\bar{Y}_s^i}\right) ds+\bar{K}_T^i-\bar{K}_t^i-\int_t^T \int_{E} \bar{U}_s^i(e) \tilde{\mu}^i(d s, d e), \forall t \in[0, T], \quad \mathbb{P} \text {-a.s. }\\
\bar{Y}_t^i \geq h\left(t, \bar{Y}_t^i, \mathbb{P}_{\bar{Y}_t^i}\right), \quad \forall t \in[0, T], \quad \mathbb{P} \text {-a.s. }\\
\int_0^T\left(\bar{Y}_{t}^i-h\left(t, \bar{Y}_{t}^i, \mathbb{P}_{\bar{Y}_{t}}\right)\right) d K_t^i=0 \quad \mathbb{P} \text {-a.s.}
\end{array}\right.
\end{equation}

\begin{thm}
Assume assumptions (H2'), (H3') and (H4') hold. Define $ \Delta Y^i:=Y^i-\bar{Y}^i$ for $1 \leq i \leq N$. Then
$$
\sup _{0 \leq t \leq T} E\left[\left|\Delta Y^i\right|^2\right]=\mathcal{O}\left(N^{-1 / 2}\right).
$$
\end{thm}
\begin{proof}
From equation (\ref{eq EsupY leq law}) in the proof of Proposition \ref{prop cvg rate}, we can deduce that:
\begin{equation}
\sup _{0 \leq t \leq T} \mathbb{E}\left[\left|Y_t^{i}-\bar{Y}_t^i\right|^2\right] \leq \sup _{0 \leq t \leq T} \mathbb{E}\left[e^{2 \beta T}\left|Y_t^{i}-\bar{Y}_t^i\right|^2\right] \leq \frac{e^{K_2}}{\lambda} \mathbb{E}\left[\Gamma_{n, 2}\right],
\end{equation}
where $\lambda:=1-8\left(\gamma_1^2+\gamma_2^2\right)>0$, $K_2:=\frac{1}{\lambda} 2 \eta T C_f^2 $ and
$$
\Gamma_{n, 2}:=\left(4 \gamma_2^2+2 \eta  C_f^2\right) \sup _{0 \leq s \leq T} e^{2 \beta T} \mathcal{W}_1^2\left(L_n\left[\bar{\mathbf{Y}}_s\right], \mathbb{P}_{Y_s}\right).
$$
We recall the bound of the right hand side of the previous inequality from \cite{briand2021particles,fournier2015rate}. Based on the regularity in Proposition \ref{prop regularity of mean-field}, we have:
$$
\sup _{0 \leq t \leq T} E\left[\left|\Delta Y^i\right|^2\right] \leq \frac{C}{\sqrt{N}}
$$
with $C$ depending all parameters.
Hence, we deduce that $\sup _{0 \leq t \leq T} E\left[\left|\Delta Y^i\right|^2\right]=\mathcal{O}\left(N^{-1 / 2}\right)$.
\end{proof}

\bibliographystyle{plain}
\bibliography{RBSDEJ}

\end{document}